\documentclass[a4paper,10pt]{amsart}
\usepackage[utf8]{inputenc}
\usepackage[T1]{fontenc}
\usepackage[style=numeric,%
	hyperref,%
	doi=false,%
	url=false,%
	isbn=false,%
	backend=bibtex%
	]{biblatex}
\bibliography{./Article-2018-03-25.bib}
\usepackage{lscape}

\usepackage{amssymb}
\usepackage{mathrsfs}
\usepackage{hyperref}
\usepackage[usenames,dvipsnames]{xcolor}
\hypersetup{colorlinks,%
citecolor=Black,%
filecolor=Black,%
linkcolor=Black,%
urlcolor=Black}
\usepackage{enumitem}	% Per personalizzare gli elenchi
\usepackage{thmtools}

\theoremstyle{plain}
\newtheorem{proposition}{Proposition}[section]
\newtheorem{theorem}[proposition]{Theorem}
\newtheorem*{theorem*}{Theorem}
\newtheorem{thm}{Theorem}[section]

\newtheorem{lemma}[proposition]{Lemma}
\newtheorem{corollary}[proposition]{Corollary}

\newtheorem*{conjecture}{Conjecture}
\theoremstyle{definition}

\declaretheorem[name=Remark,sibling=proposition,qed={\raisebox{+0.5ex}{\hbox{\small$\blacklozenge$}}}]{remark}

\newcommand{\frk}[1]{\mathfrak{#1}}
\newcommand{\bb}[1]{\mathbb{#1}}

\newcommand{\N}{\mathbb{N}}	% Numeri naturali
\newcommand{\Z}{\mathbb{Z}}	% Numeri interi
\newcommand{\R}{\mathbb{R}}	% Numeri reali
\newcommand{\Id}{\mathrm{Id}}	% Mappa identit√†
\newcommand{\Span}{\mathrm{span}}	% Span
\newcommand{\de}{\partial}		% Derivata parziale
\newcommand{\into}{\hookrightarrow}		% (->  %)

\newcommand{\ddx}{\framebox[\width]{$\Rightarrow$} }
\newcommand{\ssx}{\framebox[\width]{$\Leftarrow$} }

\newcommand{\Ad}{\mathrm{Ad}}
\newcommand{\ad}{\mathrm{ad}}
\newcommand{\Aut}{\mathtt{Aut}}
\newcommand{\cont}{\mathtt{Cont}}

\title[Polarised Lie groups]{Polarised Lie groups contactomorphic\\ to stratified groups}
\author[Nicolussi Golo]{Sebastiano Nicolussi Golo}
	\address[Nicolussi Golo]{Dipartimento di Matematica, Universit\`a di Padova, Via Trieste~63, I-35121 Padova, Italy, Orcid ID: https://orcid.org/0000-0002-3773-6471}
	\thanks{S.N.G.~has been partially supported 
	by the European Unions Seventh Framework Programme, Marie Curie Actions-Initial Training Network, under grant agreement n. 607643, ``Metric Analysis For Emergent Technologies (MAnET)'',
	and
	by the EPSRC Grant "Sub-Elliptic Harmonic Analysis" (EP/P002447/1),
	and
	by University of Padova STARS Project "Sub-Riemannian Geometry and Geometric Measure Theory Issues: Old and New".}

\author[Ottazzi]{Alessandro Ottazzi}
	\address[Ottazzi]{School of Mathematics and Statistics\\ University of New South Wales\\ UNSW Sydney 2052\\ Australia}
	\thanks{A.O.~has been partially supported by the ARC Discovery grant~DP170103025.}

\date{\today}

\subjclass[2010]{%
30L10 %Quasiconformal mappings in metric spaces
22E25, %Nilpotent and solvable Lie groups
53C30, % Homogeneous manifolds
}
\keywords{%
Tanaka prolongation, %
Stratified Lie groups,
Contact structures, %
Quasi-conformal maps %
}

\begin{document}
\begin{abstract}
For a stratified  group $G$, we construct a class of polarised Lie groups, which we call modifications of $G$, that are locally contactomorphic to it. Vice versa, we show that if a polarised group is locally contactomorphic to a stratified group $G$, whose Lie algebra has finite Tanaka prolongation, then it must be a modification of $G$.
\end{abstract}
\maketitle
\setcounter{tocdepth}{2}
\phantomsection
\addcontentsline{toc}{section}{Contents}
\tableofcontents

%%%%%%%%%%%%%%%%%%%%%%%%%%%%%%%%%%%%%%%%%%%%%%%%%%%%%%%%%%%%%%%
\section{Introduction}
In this article, we consider the following question: given a stratified group $G$, we wish to characterise those polarised Lie groups that are locally contactomorphic to $G$. 
Here a polarisation on a Lie group is the choice of a left-invariant and bracket-generating subbundle of the tangent bundle, and a contactomorphism between polarised Lie groups is a diffeomorphism that preserves the polarisation. Stratified groups carry a canonical polarisation. Given a stratified group, we will construct a class of polarised Lie groups that are contactomorphic to $G$, which we will call {\it modifications} of $G$. The key tool for our construction
will be Tanaka prolongation theory. 
\\

Before diving in the technical details of our main results, we first provide some framework.
The problem under study is relevant in different areas, such as Tanaka prolongation theory, CR geometry, sub-Riemannian geometry, and control theory. In the setting of Tanaka's theory, it is known that the infinitesimal automorphisms of the polarisation associated to a stratified Lie algebra are encoded by its full Tanaka prolongation (see, e.g., \cite{MR2917692, MR0266258, MR1274961}). If the Lie algebra is not stratified, however, all we can conclude  is that every infinitesimal automorphism induces an infinitesimal automorphism on its stratified symbol. In this paper, we construct classes of polarised Lie algebras that are not stratified, but that have the same space of infinitesimal automorphisms  as their stratified symbol.

Our study has potential applications to geometric control theory. Given a nonholonomic control system, 
 the motion planning problem consists in finding a curve tangent to the polarisation that connects two given points in the ambient space. Nilpotent Lie groups are the widest class of nonholonomic systems for which an exact solution to the motion planning problem is known, see \cite{MR3308372}. Contactomorphisms are equivalences of motion planning problems. Thus, our method detects classes of non-nilpotent nonholonomic systems that are equivalent to nilpotent ones.

Furthermore, our findings have consequences  in metric geometry.
On a polarised Lie group, one may define a left-invariant sub-Riemannian distance. In metric geometry, it is natural to study the  equivalence of metric spaces up to
isometries,  bi-Lipschitz mappings, conformal and quasiconformal mappings. For example, if two stratified groups are (locally) quasiconformal, then their Lie algebras are isomorphic (\cite{MR979599}). If two nilpotent Lie groups are isometric, then they are isomorphic (\cite{MR3646026,MR3441517}).
It is an open question to determine whether two nilpotent Lie groups that are globally bi-Lipschitz to one another are indeed isomorphic.
In~ \cite{2017arXiv170509648C}, the authors study the Lie groups that can be made isometric to a given nilpotent Lie group, endowed with a left-invariant distance. (See also \cite{MR936815} for the Riemannian case.) 
In this sense, our work follows \cite{2017arXiv170509648C}, because contactomorphisms are locally bi-Lipschitz.

In sub-Riemannian geometry, one of the major open problems is to determine whether the conclusions of Sard Theorem hold for the endpoint map, which is a canonical map from an infinite dimensional path space to the underlying finite dimensional manifold. The set of critical values for the endpoint map is also known 
as abnormal set, being the set of endpoints of abnormal extremals leaving the base point. In the context of Lie groups, perhaps the most general positive results have been proved in \cite{MR3569245}. Here the authors prove that the abnormal set has measure zero in the case of 2-step stratified groups and several other examples. This property for the abnormal set is preserved by contactomorphisms between sub-Riemannian manifolds. It then comes out from our results that every modification of a stratified group satisfies the Sard Theorem, if the stratified group does.
\\

Now we will present our main results in detail.
Recall that a Tanaka prolongation of a stratified Lie algebra $\frk g$ through a Lie subalgebra $\frk g_0$ of the Lie algebra of derivations of $\frk g$ that preserve the stratification is the maximal graded Lie algebra that contains $\frk g + \frk g_0$. When $\frk g_0$ is chosen as the whole set of strata preserving derivations, we obtain the full Tanaka prolongation.
When the prolongation algebra is finite dimensional, we obtain a graded Lie algebra $\frk p= \frk{g}\oplus \frk{q}$. Modifications of $\frk g$ are then defined to be subalgebras $\frk s$ of $\frk p$ of the same dimension of $\frk g$ that are transversal to $\frak q$. It turns out that there is a linear map $\sigma: \frk g\to\frk q$ of which the modification is the graph. If $V$ is the first layer of $\frk g$, then the set $\{v+\sigma(v)\,:\, v\in V\}$ defines a polarisation on a Lie group whose Lie algebra is the modification $\frk s$. 

Our first main result draws the connection between modifications of $G$ and Lie groups that are locally contactomorphic to $G$.
\begin{thm}\label{thm5ea02b09}
	Every modification of a stratified Lie group $G$ is locally contactomorphic to 
$G$.
	Viceversa, if $G$ is rigid, then every polarised Lie group that is locally contactomorphic to $G$ is one of its modification.
\end{thm}

Theorem~\ref{thm5ea02b09} is restated and proven in  Theorems~\ref{thm:mainone} and~\ref{teo03291510}.
A key tool in the study of local contactomorphisms is the quotient manifold $M=P/Q$, where $P$ and $Q<P$ are the Lie groups with Lie algebras $\frk p$ and $\frk q$ respectively.
The polarisation of $G$ induces a polarisation $\Delta_M$ on $M$ and $G$ embeds in $M$ as an open subset, see Proposition~\ref{prop03271803}.
Moreover, if $S$ is a modification of $G$, then an open neighborhood of the identity in $S$ can be also embedded into $M$, see Lemma~\ref{lem04071119}.
The composition of such embeddings induce a local contactomorphism between $G$ and $S$.
If the group $G$ is rigid, i.e., its full tanaka prolongation is finite dimensional, then all contactomorphisms between $G$ and $S$ arise in this way, see Theorem~\ref{teo03291510}. 

The rigid case is particularly favorable because all contact diffeomorphisms of $G$ are induced by affine maps on $P$, see~\eqref{eq5e904c47} at page~\pageref{eq5e904c47}.
We can express this rigidity in terms of local contactomorphisms of $M$.
More precisely, we will prove in Theorem~\ref{thm5e905026} the following statement:

\begin{thm}\label{thm5ea02f5e}
	Suppose $G$ is rigid and let $M=P/Q$ be the manifold described above, where the Lie algebra of $P$ is the full tanaka prolongation of $\frk g$.
	Let $U\subset M$ be open and connected and $f:U\to f(U)\subset M$ a smooth map with $df(\Delta_M)\subset \Delta_M$. 
	Suppose that there exists $x_0\in U$ such that $df(x_0)$ is non-singular.
	Then there exists a unique contact diffeomorphism $g:M\to M$ such that $g|_U=f$.
\end{thm}

We will also prove that, in the hypothesis of Theorem~\ref{thm5ea02f5e}, the connected component of the identity in the group of contact diffeomorphisms of $M$ is isomorphic to $Q$, see Theorem~\ref{thm5eb588e6}.

It remains open whether the second part of Theorem~\ref{thm5ea02b09} holds true without asking that the full Tanaka prolongation is finite. While we cannot prove the theorem in this generality, examples suggest that it may be true. More precisely,  in Subsection~\ref{heisenberg}, we show that all three dimensional sub-Riemannian structures are modifications of the Heisenberg group with respect to a suitable finite dimensional Tanaka prolongation, even though  the full prolongation of the Heisenberg Lie algebra is infinite dimensional, see Theorem~\ref{thm5ea01a1d}.
This justifies the following conjecture.

\begin{conjecture}
Suppose that $G$ is a stratified Lie group and that $S$ is a polarised Lie group that is locally contactomorphic to $G$. Then there is a finite Tanaka prolongation of ${\rm Lie}(G)$ in which ${\rm Lie}(S)$ is a modification of ${\rm Lie}(G)$.
\end{conjecture}

In Subsection~\ref{freemod} we explicitly compute some modifications of the free nilpotent Lie group with two generators and step four, $F_{24}$. It comes out that one may construct  examples of non-nilpotent Lie groups that are contactomorphic to  $F_{24}$. We also find a nilpotent, non-stratified, polarised Lie group that is globally contactomorphic to $F_{24}$, see Theorem~\ref{thm5ea02e50}.
Finally, in Subsection~\ref{ultra}, we study all the modifications of an ultra-rigid stratified group, that is, a stratified group whose only strata-preserving derivation is the infinitesimal generator of dilations. It turns out that such modifications are all solvable and the only nilpotent one is the stratified group itself, see Theorem~\ref{thm5ea02e33}.

The paper is organized as follows.
In Section~\ref{prel}, we fix the notation and establish the framework in which we will be working. We consider stratified  algebras and their Tanaka prolongations, we define the corresponding Lie groups and fix a polarisation on them.
In Section~\ref{sec5ea02faf}, we study contactomorphisms of $M$ as affine maps of $P$ and prove Theorem~\ref{thm5ea02f5e}.
In Section~\ref{mod}, we define the modifications of a stratified algebra and those of a stratified group, proving Theorem~\ref{thm5ea02b09}.
Finally, we apply our modification technic to a number of examples in Section~\ref{exsection}.

%%%%%%%%%%%%%%%%%%%%%%%%%%%%%%%%%%%%%%%%%%%%%%%%%%%%%%%%%%%%%%%
%%%%%%%%%%%%%%%%%%%%%%%%%%%%%%%%%%%%%%%%%%%%%%%%%%%%%%%%%%%%%%%
\section{Notation and preliminaries}\label{prel}
\subsection{Polarizations and Tanaka prolongations}
Given a connected, smooth manifold $M$, a {\it polarisation} of $M$ is the choice of a subbundle $\Delta_M$ of the tangent bundle $TM$ that is \emph{bracket generating}, i.e., with the property that the sections of $\Delta_M$ bracket generate all the sections of $TM$. Given two polarised manifolds $(M,\Delta_M)$ and $(N,\Delta_N)$, a \emph{contactomorphism} between 
$M$ and $N$  is a diffeomorphism $f:M\to N$ such that 
$f_*(\Delta_M)=\Delta_N$.
We denote by $\Gamma(TM)$ the space of vector fields on $M$.
A vector field $V\in\Gamma(TM)$ on a polarised manifold $(M,\Delta_M)$ is a \emph{contact vector field} if its flow is made of contactomorphisms.
For a Lie group $S$, we shall always consider left-invariant polarisations $\Delta_S$. The pair $(S,\Delta_S)$ is called a {\it polarised group}. 
The identity element will be denoted by $e_S$, or simply $e$ if no confusion arises.
 We  denote by $G$ a \emph{stratified group}, that is, a connected and simply connected  Lie group whose
 Lie algebra decomposes as $\frk g=\bigoplus_{i=-s}^{-1}\frk g_{i}$, with $[\frk g_{-1},\frk g_{j}]=\frk g_{j-1}$ for every $-s+1\leq j\leq -1$. On a stratified group we will always consider the left-invariant polarisation $\Delta_G$ for which $(\Delta_G)_{e_G}=\frk g_{-1}$.
In  a stratified group $G$ we consider the strata preserving derivations
\begin{align*}
{\rm Der}(\frk g):= \{u\in {\rm End}({\frk g}) & \,:\, u({\frk g}_{-1})\subset  {\frk g}_{-1}, \\
  \text{ and } & u[X,Y]=[u(X),Y]+[X,u(Y)]\, \forall X,Y\in  {\frk g}\}.
\end{align*}
Given  a subalgebra ${\frk g}_0$ of  ${\rm Der}(\frk g)$, we define the \emph{Tanaka prolongation} of ${\frk g}$ through ${\frk g}_{0}$ as the (possibly infinite) maximal graded Lie algebra ${\rm Prol}({\frk g}, {\frk g}_{0})=\bigoplus_{k\geq -s} {\frk g}_{k}$ which contains $\frk g\oplus \frk g_0$. When ${\frk g}_{0}={\rm Der}(\frk g)$, we  call ${\rm Prol}({\frk g}, {\frk g}_0)$ the {\it full} Tanaka prolongation of ${\frk g}$. It is not difficult to see that the latter contains all prolongations. 
We say that $\frk g$, or $G$, is {\it rigid} if the full Tanaka prolongation has finite dimension.
When it is clear from the context and the prolongation under consideration is finite dimensional, we shall denote ${\rm Prol}({\frk g}, {\frk g}_{0})$  by ${\frk p}$, 
the nonnegative part $\bigoplus_{k\geq 0} {\frk g}_{k}$ by $\frk q$, and the positive part
$\bigoplus_{k> 0} {\frk g}_{k}$ by ${\frk p}_+$. See \cite{MR2917692,MR0266258, MR1274961} for further details on Tanaka prolongation. 

%%%%%%%%%%%%%%%%%%%%%%%%%%%%%%%%%%%%%%%%%%%%%%%%%%%%%%%%%%%%%%%
\subsection{The groups $P$ and $Q$ and their quotient $M$}
In the following, we establish a number of properties of the Lie groups that correspond to the Lie algebras introduced above. Let $\bar P$ be the connected and simply connected Lie group whose Lie algebra is a finite dimensional Tanaka prolongation $\frk p$ of a stratified Lie algebra $\frk g$. 
Let $\bar Q$ be the connected subgroup of $\bar P$ whose Lie algebra is $\frk q$.

The set $\{ \delta_\lambda : \lambda>0\}$ of mappings on $\frk p$ defined by $ \delta_\lambda(X)=\lambda^i X$ for $X\in \frk g_i$ is 
 a one parameter family of automorphisms of $\frk p$. By abuse of notation, we write $\delta_\lambda$ for the corresponding automorphisms of the group $\bar P$.
Such maps exist because $\bar P$ is simply connected.
 
\begin{lemma}\label{lem11121335} 
	Denote by $\exp_P:\frk p\to \bar P$ the exponential map of $\bar P$.
	Then $\exp_P$ is injective on $\frk g$ and on $ì\bigoplus_{k\geq 1} {\frk g}_{k}$. 
\end{lemma}
\begin{proof}
	Let $v,w\in\frk g$ such that $\exp_P(v)=\exp_P(w)$.
	Since $v,w\in\frk g$, then $\lim_{\lambda\to\infty}\delta_\lambda v =\lim_{\lambda\to\infty}\delta_\lambda w = 0$.
	Let $\lambda\ge1$ be such that both $\delta_\lambda(v)$ and $\delta_\lambda(w)$ belong to a neighborhood $U$ of $0$ in $\frk p$ on which the exponential map $\exp_P$ is injective.
	Then $\exp_P(\delta_\lambda v) = \delta_\lambda(\exp_P(v)) = \delta_\lambda(\exp_P(w)) = \exp_P(\delta_\lambda w)$.
	By the injectivity of $\exp_P$ on $U$, we have $\delta_\lambda v = \delta_\lambda w$.
	Since $\delta_\lambda$ is a linear isomorphism, we conclude that $v=w$.
A similar argument proves that $\exp_P$ is injective on $\bigoplus_{k\geq 1} {\frk g}_{k}$. 
\end{proof}

By Lemma~\ref{lem11121335}, the canonical immersion $G\into \bar P$ induced by $\frk g\into \frk p$ is injective.
We are going to show that $G$ is closed in $\bar P$. We prove two lemmas first.

\begin{lemma}\label{lem11121630}
	The intersection of $G$ with $\bar Q$ is trivial.
\end{lemma}
\begin{proof}
	Since $\delta_\lambda(\frk g)=\frk g$ and $\delta_\lambda(\frk q)=\frk q$, then $\delta_\lambda(G)= G$ and $\delta_\lambda(\bar Q)=\bar Q$, for all $\lambda>0$.
	Since $\frk g$ is nilpotent, $G=\exp_P(\frk g)$.
	
	Let $x\in  G\cap \bar Q$; then $x=\exp_P(v)$ for some $v\in\frk g$ and $\lim_{\lambda\to\infty}\delta_\lambda(x) = \exp_P(\lim_{\lambda\to\infty}\delta_\lambda v) = e_P$.
	It follows that the curve $\gamma:(0,1]\to \bar P$, $\gamma(t) = \delta_{t^{-1}}x$, extends to a continuous path $[0,1]\to \bar P$ connecting $\gamma(0)=e_P$ to $\gamma(1)=x$ and laying in $G$.
	Since $\delta_\lambda(x)\in \bar Q$ for all $\lambda>0$, then $\gamma$ lies in $\bar Q$ as well.
	 	 
	 Since $\frk g\oplus\frk q=\frk p$, there are open neighborhoods $U\subset \frk g$ and $V\subset \frk q$ of $0$ such that $\Omega=\exp_P(U)\exp_P(V)$ is an open neighborhood of $e_P$ in $\bar P$ and the following holds:
	 The connected component of $\Omega\cap G$ containing $e_P$ is $\exp_P(U)$,
	 the connected component of $\Omega\cap \bar Q$ containing $e_P$ is $\exp_P(V)$, 
	 and $\exp_P(U)\cap\exp_P(V)=\{e_P\}$.
	 
	 Since $\gamma$ joins $x$ to $e_P$ continuously, then $\gamma([0,1])\cap \Omega$ lies in both the connected components of $\Omega\cap G$ and $\Omega\cap \bar Q$ containing $e_P$, i.e., $\gamma([0,1])\cap \Omega \subset \exp_P(U)\cap\exp_P(V)=\{e_P\}$.
	This implies that $x=e_P$.
\end{proof}

\begin{lemma}[Lemma on Lie groups]\label{lem11121512}
Let $G$ be a Lie subgroup of a Lie group $P$ and let $\iota : G\into P$ the inclusion.	
	The image $\iota(G)$ is not closed in $P$ if and only if there is a sequence $\{g_n\}_{n\in\N}\subset G$ such that $\lim_{n\to\infty} g_n = \infty$ and $\lim_{n\to\infty} \iota(g_n) = e_P$.
\end{lemma}
\begin{proof}
	Recall that $G$ is closed in $P$ if and only if $\iota$ is an embedding.
	So, if such a sequence exists then $\iota(G)$ is not closed in $P$.
	We need to prove the converse implication.

	Let   $\rho$ be any left-invariant Riemannian distance on $G$.
	Then $\rho$ is complete and in particular closed balls are compact.
	Let $\{g_n\}_{n\in\N}\subset G$ be a sequence such that $\lim_{n\to\infty} \iota(g_n) = p\in P$.
	If there is $R>0$ such that $\rho(e_G,g_n)\le R$ for all $n$, then there is a subsequence $g_{n_k}$ converging to some $g_\infty\in G$.
	Since the immersion $\iota : G\into P$ is continuous, we obtain $\iota(g_\infty)=p$, hence $p\in\iota(G)$.
	
	So, if $\iota(G)$ is not closed, then there is a sequence $\{g_n\}_{n\in\N}\subset G$ such that $\lim_{n\to\infty} \iota(g_n) = p\in P$ but $g_n\to\infty$ in $G$.
	Let $\{g_{n_k}\}_k$ be a subsequence such that $\rho(g_{n_k},g_{n_{k+1}}) > k$ for $k\in\N$ and define $h_k=g_{n_k}^{-1}g_{n_{k+1}}$.
	Then $h_k\to\infty$ in $G$, because $\rho(e_G,h_k) = \rho(e_G,g_{n_k}^{-1}g_{n_{k+1}}) = \rho(g_{n_k},g_{n_{k+1}})>k$ for all $k$.
	However, $\iota(h_k) = \iota(g_{n_k}^{-1})\iota(g_{n_{k+1}}) \to p^{-1}p$ in $P$ as $k\to\infty$.
\end{proof}

\begin{lemma}\label{lem03252021}
	The immersed group $G$ is closed in $\bar P$.
\end{lemma}
\begin{proof}
	We  prove that, if $\{v_n\}_{n\in\N}\subset\frk g$ is a sequence so that $\lim_{n\to\infty}\exp_P(v_n)=e_P$, then $\lim_{n\to\infty}v_n=0$.
	By Lemma~\ref{lem11121512} and $\exp_P(\frk g)=G$, this claim implies that $G$ is closed in $P$.

Let $\{v_n\}_{n\in\N}\subset\frk g$ be a sequence with $\lim_{n\to\infty}\exp_P(v_n)=e_P$.
	Let $U\subset\frk g$ and $W\subset\frk q$ be open neighborhoods of $0$ such that the map $U\times W\to P$, $(u,w)\mapsto \exp_P(u)\exp_P(w)$ is a diffeomorphism onto its image.
	Then, for $n$ large enough, there are $u_n\in U$ and $w_n\in W$ so that $\exp_P(u_n)\exp_P(w_n)=\exp_P(v_n)$.
	Therefore, $\exp_P(u_n)^{-1}\exp_P(v_n) = \exp_P(w_n)\in \bar Q\cap G$.
	By Lemma~\ref{lem11121630}, we have $\exp_P(u_n)=\exp_P(v_n)$.
	By Lemma~\ref{lem11121335}, we have $u_n=v_n$.
	Since $\exp_P(u_n)\to e_P$, then $v_n=u_n\to 0$.
\end{proof}

\begin{corollary}
	The immersed group $\bar Q$ is closed in $\bar P$.
\end{corollary}
\begin{proof}
	This is a consequence of
	Lemma~\ref{lem03252021} and part (iii) of Lemma~2.15 in \cite{2017arXiv170509648C}
\end{proof}

Since $\bar Q$ is closed, we may consider the homogeneous manifold $M:=\bar P/\bar Q$ with quotient projection $\pi:\bar P\to M$.
The action of $\bar P$ may have a non-trivial kernel
\[
K :=\{p\in \bar P:p.x=x\ \forall x\in M\} 
  = \bigcap_{p\in\bar P}p\bar Qp^{-1} .
\]
\begin{lemma}\label{lem04061839}
	The kernel $K$ of the action of $\bar P$ on $M$ is discrete and contained in $\bar Q$.
	Moreover, if $p\in K$, then $\delta_\lambda p=p$ for all $\lambda>0$.
\end{lemma}
\begin{proof}
	Clearly $K$ is a normal and closed subgroup of $\bar P$ and it is contained in $\bar Q$.
	Let $v\in {\rm Lie}(K)$, the Lie algebra of $K$. Then for some positive integer $\ell$, we may write $v= v_0+\dots+v_\ell$, with $v_i\in \frk{g}_i$ for every $i=0,\dots,\ell$. Since ${\rm Lie}(K)$ is an ideal in $\frk p$ contained in $\frk q$, it follows in particular that for all $i=0,\dots,\ell$,
	\[
	[[\dots[[v_i,y_1],y_2],\dots],y_{\ell+1}]\in \frk{g}_{i-\ell-1}\cap \frk{q}=\{0\},
	\]
	for every $y_1,\dots,y_{\ell+1}\in \frk{g}_{-1}$. By definition of Tanaka prolongation, this implies that $v=0$.
	Therefore, the Lie algebra of $K$ is trivial and so $K$ is discrete.
	
	Since $K=\bigcap_{x\in \bar P} x\bar Qx^{-1}$, it is clear that $\delta_\lambda(K)\subset K$ for all $\lambda>0$.
	However, since $\lambda\mapsto \delta_\lambda p$ is a continuous curve passing through $p$, we must have $\delta_\lambda p=p$ when $p\in K$.
\end{proof}
From Lemma~\ref{lem04061839} it follows that $P:=\bar P/K$ and $Q:=\bar Q/K$ are Lie groups, that $Q$ is closed in $P$ and $M=P/Q$.
Moreover, the maps $\delta_\lambda$ are automorphisms of $P$ as well, for all $\lambda>0$.
Since $G\cap K=\{e\}$, the group $G$ is  embedded in $P$ with $G\cap Q=\{e\}$.

\begin{remark}\label{rem05101810}
	If we are given $G$ and $Q$ inside $P$, for instance as matrix groups, we may want to visualise the action of $P$ on $M$ as a local action of $P$ on $G$.
	In other words, if $p\in P$, then there may be open subsets $U_p,V_p\subset G$ and a contactomorphism $f_p:U_p\to V_p$ that corresponds to the action of $p$ on $M$, i.e., $f_p(g_1)$ is the only $g_2\in G$, if it exists, such that $\{g_1Q)\cap G = \{g_2\}$.
	In general, such construction is not possible for all $p\in P$, but if $p$ is near enough to $e_P$, then $U_p$, $V_p$ and $f_p$ do exist.
	The fact that such $f_p$ is a contactomorphism will be proved in Proposition~\ref{prop03271653}.
\end{remark}

%%%%%%%%%%%%%%%%%%%%%%%%%%%%%%%%%%%%%%%%%%%%%%%%%%%%%%%%%%%%%%%
\subsection{Polarizations on $G$, $P$ and $M$}
We denote by $\pi:P\to M$ the quotient map, with $M=P/Q$.
If $p\in P$ and $m\in M$, we use the notation $p.m$ or $p(m)$ for the action of $p$ on $m$.
In such contexts, we will identify elements $p\in P$ with smooth diffeomorphisms $p:M\to M$.

Recall that on $G$ we have the polarisation $\Delta_G$ with $(\Delta_G)_{e}=\frk g_{-1}$.
We define on $P$ the polarisation $\Delta_P$ such that $(\Delta_P)_{e_P}=\frk g_{-1} \oplus \frk q$.
Notice that $\Delta_G=\Delta_P\cap TG$.
Define $\Delta_M:=d\pi (\Delta_P)$ which is a subset of $TM$.
We shall prove that $\Delta_M$ is a $P$-invariant polarisation on $M$.

\begin{proposition}\label{prop03271653}
	The set $\Delta_M\subset TM$ is a $P$-invariant, bracket generating subbundle of $M$.
	In particular, $(M,\Delta_M)$ is a polarised manifold and the diffeomorphisms $p:M\to M$ for $p\in P$ are contactomorphisms. 
\end{proposition}
\begin{proof}
	Notice that $\Delta_M$ is a $P$-invariant subset of $TM$.
	In order to show that $\Delta_M$ is a subbundle, we need to prove that, if $p_1,p_2\in P$ are such that $\pi(p_1)=\pi(p_2)$, then 
	\begin{equation}\label{eq03252104}
	d\pi((\Delta_P)_{p_1}) = d\pi((\Delta_P)_{p_2}) .
	\end{equation}
	Since $p\circ\pi=\pi\circ L_p$ for all $p\in P$, then~\eqref{eq03252104} is equivalent  
	to $d(p_2^{-1}\circ\pi\circ L_{p_1})[(\Delta_P)_{e}] = d\pi[(\Delta_P)_{e}]$.
	Let $p=p_1$ and $q\in Q$ such that $p_2=p_1q$.
	Then $p_2^{-1}\circ\pi\circ L_{p_1} = \pi\circ L_{q^{-1}}$ and thus~\eqref{eq03252104} is also equivalent to
%	PROOF OF THE STATEMENT:
%	\footnote{
%	$
%	D((pq)^{-1})\circ D\pi\circ DL_p (\Delta_e) = D(\pi\circ L_{pq}^{-1}\circ L_p)(\Delta_e)
%	= D\pi (D L_{q^{-1}} DR_{q} DR_{q^{-1}}(\Delta_e))
%	= D\pi (\Ad_q(\Delta_e)) .
%	$
%	}
	\begin{equation}\label{eq01051217}
	\Ad_q[(\Delta_P)_{e}] \mod\frk q = (\Delta_P)_{e} \mod\frk q .
	\end{equation}
	Since $\Ad$ is a homomorphism and every $q\in Q$ is the finite product of exponential elements, it's enough that we show \eqref{eq01051217} for $q=\exp y$, $y\in \frk q$. Denote by $y_0$ the projection of $y$ on $\frk g_0$.
	Let $w\in \frk g_{-1} \oplus \frk q$ and denote by $w_{-1}$ its projection on $\frk g_{-1}$.
	Then 
	\begin{align*}
	\Ad_qw \mod\frk q
	&= e^{\ad(y)} w \mod\frk q \\
	&= e^{\ad(y_0)} w_{-1} \mod\frk q .
	\end{align*}
	Since $e^{\ad(y_0)} :\frk g_{-1}\to \frk g_{-1}$ is a bijection, 
we conclude that $\Ad_q[(\Delta_P)_{e}] \mod\frk q= \frk g_{-1} \mod \frk q$.
	This proves \eqref{eq01051217} and therefore~\eqref{eq03252104}.
	
	Finally, we need to show that $\Delta_M$ is bracket generating.
	Remind that, for an analytic subbundle of an analytic manifold, being bracket generating is equivalent to being connected by curves tangent to the subbundle, and that quotients of Lie groups and invariant subbundles are all analytic.
	Thus, let $m_0=\pi(p_0)$ and $m_1=\pi(p_1)$ in $M$.
	Then there is a $C^1$-curve $\gamma:[0,1]\to P$ such that $\gamma(0)=p_0$, $\gamma(1)=p_1$ and $\gamma'(t)\in \Delta_P$ for all $t\in[0,1]$.
	Hence, the curve $\pi\circ\gamma:[0,1]\to M$ goes from $m_0$ to $m_1$ and is clearly tangent to $\Delta_M$.
\end{proof}

\begin{proposition}\label{prop03271803}
	The restriction $\pi|_G:(G,\Delta_G)\to (M,\Delta_M)$ is a contactomorphism onto its image, which is an open subset of $M$.
\end{proposition}
\begin{proof}
	First, we show that $\pi|_G$ is injective.
	Let $a,b\in G$ such that $\pi(a)=\pi(b)$.
	Then $\pi(e) = \pi(a^{-1}a) = a^{-1}.\pi(a) = a^{-1}\pi(b) = \pi(a^{-1}b)$, i.e., $a^{-1}b\in Q$.
	Since $a^{-1}b\in G$ and $G\cap Q=\{e\}$, then $a=b$.
	
	Second, we show that $\pi|_G$ is an immersion.
	Since $\ker(d\pi|_g) = dL_g(\frk q)$ and $dL_g(\frk q)\cap T_gG=DL_g(\frk q)\cap DL_g(\frk g)=\{0\}$, then $d(\pi|_G)|_g = (d\pi|_g)|_{T_gG}$ is injective, for all $g\in G$.
%	ARGOMENTO LUNGO E NOIOSO QUI SOTTO.
%	Fix $g\in G$ and set $\psi_g(x):=\pi(gx)=g.\pi(x)$.
%	Then we have, for all $y\in G$, $\pi(y) = \pi(gg^{-1}y) = \psi_g\circ L_{g^{-1}}(y)$.
%	Therefore, 
%	\[
%	(D\pi|_G)|_g 
%	= D(\psi_g\circ L_{g^{-1}})|_g 
%	= Dg|_{\pi(e)} \circ D\pi|_e \circ DL_{g^{-1}}|_g ,
%	\]
%	which is a surjective map $T_gG\to T_{\pi(g)}M$, because $\ker(D\pi|_e) = \frk q$.
	
	Third, we  claim $d\pi|_G(\Delta_G) = \Delta_M\cap T(\pi(G))$.
	Since $\Delta_G\subset \Delta_P$ and since $\Delta_M=d\pi(\Delta_P)$ by definition, it follows that $d\pi|_G(\Delta_G) \subset\Delta_M\cap T(\pi(G))$.
	Moreover, since $\Delta_M$ is $P$-invariant by Proposition~\ref{prop03271653}, for all $x\in M$ , $\dim(\Delta_M)_x=\dim(\Delta_M)_{\pi(e)}=\dim((\frk g_{-1}\oplus\frk q)/\frk q)=\dim(\frk g_{-1})$.
	Therefore, we obtain the claim by comparing the dimensions.
	
	Finally, the facts that $\pi(G)$ is open in $M$ and that $\pi|_G$ is an embedding are consequences of $(\pi|_G)$ being an immersion and the fact that $M$ and $G$ have the same dimension.
\end{proof}

\begin{remark}
	A first consequence of Proposition~\ref{prop03271803} is that any local contactomorphism on $M$ is in fact a local contactomorphism on $G$.
	Indeed, by the action of $P$ on $M$ and via the map $\pi|_G$, any local contactomorphism of $M$ defines a local contactomorphism of $G$.
	Similarly, contact vector fields on $M$ define contact vector fields on $G$.

	In case $G$ is a rigid stratified group and $\frk p$ is the full Tanaka prolongation of $\frk g$, these relations are stronger, see Section~\ref{sec5ea02faf}.
\end{remark}

%%%%%%%%%%%%%%%%%%%%%%%%%%%%%%%%%%%%%%%%%%%%%%%%%%%%%%%%%%%%%%%
%%%%%%%%%%%%%%%%%%%%%%%%%%%%%%%%%%%%%%%%%%%%%%%%%%%%%%%%%%%%%%%
\section{Contactomorphisms of $M$ when $G$ is rigid}\label{sec5ea02faf}
This section contains Theorem~\ref{thm5e905026} for contact diffeomorphisms in the rigid case.

Relative to a vector $X\in T_eP$, we denote by $\tilde X$ the  left-invariant vector field $\tilde X(p)=dL_p|_e[X]$, and by $X^\dagger$ the  right-invariant vector field $ X^\dagger(p)=dR_p|_e[X]$.
Similarly, we denote by $\tilde{\frk p}$ the Lie algebra of left-invariant vector fields and by $\frk p^\dagger$ the Lie algebra of right-invariant vector fields on $P$.
Moreover, as in the previous sections, the manifold $M$ is the quotient $P/Q$ and we denote by $o$ the point $\pi(e)\in M$.

\begin{lemma}\label{lem07041435}
	Let $\ell:\frk p\to\frk p$ be a Lie algebra automorphism with $\ell(\frk q)=\frk q$ and $\ell(\frk g_{-1}\oplus\frk q) = \frk g_{-1}\oplus\frk q$.
	Then there are a unique contactomorphic Lie group automorphism $L:P\to P$ with $L_*=\ell$ and a unique contactomorphism $L^\pi:M\to M$ with $L^\pi\circ\pi=\pi\circ L$.
\end{lemma}
\begin{proof}
	If $\ell:\frk p\to\frk p$ is a Lie algebra automorphism with $\ell(\frk q)=\frk q$, then the induced Lie group automorphism $\bar L:\bar P\to\bar P$ is such that $\bar L(K)=K$, where $K$ is the kernel of the action of $\bar P$ on $M$.
	It follows that there is a Lie group automorphism $L:P\to P$ such that $L_*=\ell$.
	
	If $L:P\to P$ is a Lie group automorphism with $L(Q)=Q$, then it is well known that there is a unique diffeomorphism $L^\pi:M\to M$ such that $L^\pi\circ\pi=\pi\circ L$.
	
	Now, suppose that $\ell(\frk g_{-1}\oplus\frk q) = \frk g_{-1}\oplus\frk q$, i.e., $L_*(\Delta_P)_e = (\Delta_P)_e$.
	Since $\Delta_P$ is left-invariant, then $L$ is a contactomorphism of $(P,\Delta_P)$.
	Finally, we prove that $L^\pi$ is a contactomorphism.
	Let $X\in\Delta_P$ and $x\in P$.
	Then
	\[
	dL^\pi|_{\pi(x)}[d\pi|_x[\tilde X_x]]
	= d(L^\pi\circ\pi)|_x[\tilde X_x]
	= d(\pi\circ L)|_x[\tilde X_x]
	\in \Delta_M|_{L^\pi(x)} .
	\]
\end{proof}

Let $\Aut(\frk p,\frk g)$ be the group of Lie algebra automorphisms of $\frk p$ that induce contactomorphism on $M$.
By Lemma~\ref{lem07041435}, we have
\[
\Aut(\frk p,\frk g)=\{\phi\in\Aut(\frk p):\phi(\frk q)=\frk q,\ \phi(\frk g_{-1}\oplus\frk q)=\frk g_{-1}\oplus\frk q\} .
\]
This group plays a crucial role in the classification of modifications of a stratified Lie algebra, as we shall show in Theorem~\ref{thm05262101}.

For the following claim, see \cite[Section~6]{MR0266258} and~\cite{MR1274961}.
\begin{theorem}[Tanaka]\label{thm05101222}
	If $\frk g$ is rigid and  $\frk p$ is the full Tanaka prolongation,
	then $\pi_*\frk p^\dagger\subset\Gamma(TM)$ is the set of all germs of contact vector fields on $M$.
	More precisely, on the one hand $\pi_*\frk p^\dagger$ are contact vector fields of $(M,\Delta_M)$; 
	On the other hand, if $U\subset M$ is open and connected, and $V\in\Gamma(TU)$ is a contact vector field, then there is a unique $X\in \frk p$ such that $V=\pi_*X^\dagger|_U$.
\end{theorem}

Denote by $\cont(U)$ the space of contact vector fields on an open set $U\subset M$.
Notice that $\cont(U)$ is a Lie algebra.
Theorem~\ref{thm05101222} can be restated as:
{\it if $U\subset M$ is open and connected, then $\pi_*|U:X^\dagger\mapsto \pi_*X^\dagger|_{U}$ is a Lie algebra isomorphism between $\frk p^\dagger$ and $\cont(U)$.}
With Theorem~\ref{thm05101222}, we can prove the following result:

\begin{theorem}\label{thm5e905026}
	Suppose $\frk g$ is rigid and let $\frk p$ be its full Tanaka prolongation.
	Let $U\subset M$ open and path-connected and $f:U\to f(U)\subset M$ a smooth map with $df(\Delta_M)\subset \Delta_M$ and suppose that there exists $x_0\in U$ such that $df(x_0)$ is non-singular.
	Then there exist unique a Lie group automorphism $L^f:P\to P$ such that $(L^f)_*\in\Aut(\frk p,\frk g)$ and $p^f\in P$ such that, for every $p\in P$ with $\pi(p)\in U$, we have
	\begin{equation}\label{eq5e904c47}
	f(\pi(p)) = \pi(p^fL^f(p)) .
	\end{equation}
	In particular, there exists a unique contact diffeomorphism $g:M\to M$ such that $g|_U=f$.
\end{theorem}

For proving Theorem~\ref{thm5e905026}, we need a preliminary lemma.

\begin{lemma}\label{lem04261852}
	Suppose $\frk g$ is rigid and let $\frk p$ be its full Tanaka prolongation.
	Let $U\subset M$ open and path-connected and $f:U\to f(U)\subset M$ a smooth diffeomorphism with $df(\Delta_M)\subset \Delta_M$.
	Then there exist unique $p_f\in P$ and a Lie group automorphism $L^f:P\to P$ with $(L^f)_*\in\Aut(\frk p,\frk g)$ such that, for every $x\in P$ with $\pi(x)\in U$, we have
	\begin{equation}\label{eq5e9f1f85}
	f(\pi(x)) = \pi(p_fL^f(x)) .
	\end{equation}
	In particular, there exists a unique contact diffeomorphism $g:M\to M$ such that $g|_U=f$.
\end{lemma}

\begin{proof}
	By Theorem~\ref{thm05101222},
	the map $\ell^f=\pi_*|_{f(U)}^{-1} \circ f_* \circ \pi_*|_{U}:\frk p^\dagger \to \frk p^\dagger$ is a composition of Lie algebra isomorphisms
	\[
	\frk p^\dagger \to \cont(U) \to \cont(f(U)) \to \frk p^\dagger .
	\]
	Let $L^f:P\to P$ be the corresponding Lie group automorphism, whose existence is assured by Lemma~\ref{lem07041435}.
	
	It is clear that, if $f_1,f_2$ are two such contact diffeomorphisms, then $L^{f_2\circ f_1} = L^{f_2}\circ L^{f_1}$, whenever the composition is well defined.
	Moreover, notice that, if $p\in P$, then, seen as a diffeomorphism $p:M\to M$, the argument above shows that $\ell^p=\Id$, and so~\eqref{eq5e9f1f85} holds with $L^p=\Id$ and $p_p=p$.
	
	Back to the general case, let $p_0,p_1\in P$ such that $\pi(p_0)\in U$ and $\pi(p_1)=f(\pi(p_0))$.
	Then $h(m):=p_1^{-1}.f(p_0.m)$ defines a contact diffeomorphism $h:p_0^{-1}.U\to p_1.f(U)$.
	By the previous paragraph, we also have $L^h=L^f$.
	Since $h(\pi(e_P))=\pi(e_P)$, integrating the contact vector fields, we get 
	$h(\pi(x)) = \pi(L^h(x))$ whenever $\pi(x)\in U$.
	We conclude that~\eqref{eq5e9f1f85} holds with $p^f:=p_1L(p_0^{-1})$.
	
	Viceversa, if $L^f$ and $p^f$ satisfy~\eqref{eq5e9f1f85}, then
	the differential $\ell_f$ of $L^f$ at $e_P$ is a Lie algebra automorphism of $\frk p^\dagger$ satisfying $\ell^f=\pi_*|_{f(U)}^{-1} \circ f_* \circ \pi_*|_{U}$.
	Therefore, $L^f$ is unique.
	Moreover, since $f(\pi(x)) = \pi(p_fL^f(x)) = p_f.\pi(L^f(x))$, also $p_f$ is uniquely determined.
	
	Finally, the fact that $x\mapsto p_fL^f(x)$ induces a global diffeomorphism $g:M\to M$, which extends $f$, is a consequence of Lemma~\ref{lem07041435}.
\end{proof}

\begin{proof}[Proof of Theorem~\ref{thm5e905026}]
	Fix $x_0\in U$ and a neighborhood $U'\subset U$ of $x_0$ such that $f$ is a diffeomorphism $U'\to f(U')$.
	By Lemma~\ref{lem04261852}, there is a contact diffeomorphism $g:M\to M$ with $g|_{U'}=f|_{U'}$.
	Let $W\subset U$ be the largest open set where $f$ and $g$ are equal.
	If $x\in \de W\cap U$ then $df(x)=dg(x)$, by continuity of the differential, and thus $df(x)$ is non singular.
	Applying Lemma~\ref{lem04261852} again, there are a neighborhood $U''$ of $x$ and a contact diffeomorphism $g'':M\to M$ such that $g''|_{U''}=f|_{U''}$.
	Notice that $U''\cap W$ is a nonempty open set and that the restrictions of $g''$ and $g$ are contact diffeomorphisms on $U''\cap W$ and $g''|_{U''\cap W}=g_{U''\cap W}$.
	Since Lemma~\ref{lem04261852} claims the uniqueness of smooth contact extensions, we get $g''=g$.
	In particular, we have that $x\in U''\subset W$, in contradiction with $x\in\de W$.

	Therefore, we conclude that $\de W\cap U=\emptyset$ and, since $U$ is connected, $W=U$.
\end{proof}

\begin{remark}
	In Theorem~\ref{thm05101222}, and consequently in Theorem~\ref{thm5e905026}, one can assume $f$ to be only smooth of class $C^2$.
	The upgrade of the regularity works like in~\cite{MR2917692}.
\end{remark}

Finally, we prove that the group of $\Aut(\frk p,\frk g)$ is the adjoint representation of $Q$.

\begin{theorem}\label{thm5eb588e6}
	Suppose that $\frk g$ is rigid and that $\frk p$ is the full Tanaka prolongation.
	The Lie algebra of $\Aut(\frk p,\frk g)$ is $\{\ad_X:X\in\frk q\}$.
	In particular, the connected component of the identity in $\Aut(\frk p,\frk g)$ is $\{\Ad_x:x\in Q\}$, which is isomorphic to $Q$ via the adjoint map $x\mapsto \Ad_x$.\end{theorem}
\begin{proof}
	We need to show that, if $D:\frk p\to\frk p$ is a derivation such that 
	$D(\frk q)\subset\frk q$ 
	and $D(\Delta_P|_e)\subset \Delta_P|_e$,
	then there is $X\in\frk q$ such that $D=\ad_X$.
	
	The one-parameter group of Lie algebra automorphisms $\ell_t:=e^{tD}$ are such that $\ell_t(\frk q)=\frk q$ and $\ell_t(\Delta_P)_e = (\Delta_P)_e$.
	By Lemma~\ref{lem07041435}, they induce a one-parameter group of Lie group automorphism $ L_t$ on $ P$ and a one-parameter group of contactomorphism $L_t^\pi:M\to M$.
	
	Since $L_t^\pi$ is a one-parameter group of contactomorphisms on $M$ and by Theorem~~\ref{thm05101222}, there is $V\in\pi_*\frk p^\dagger$ such that $L_t^\pi$ is its flow.
	Let $X\in\frk p$ be such that $\pi_*(X^\dagger)=V$.
	Since $L_t^\pi(o)=o$ and thus $V(o)=0$, we have $X\in\frk q$.
	
	Notice that $L_t^\pi(m)=\exp(tX).m$ for all $m\in M$.
	Therefore, if $p\in P$ and $m=\pi(p)$, then 
	\begin{align*}
	\pi(L_t(p)) 
	&= L_t^\pi(\pi(p)) 
	= \exp(tX).\pi(p) \\
	&= \pi(\exp(tX)p)
	= \pi(\exp(tX)p\exp(-tX))
	= \pi(C_{\exp(tX)}p) ,
	\end{align*}
	where $C_ap=apa^{-1}$ is the conjugation by $a\in P$.
	Since by Lemma~\ref{lem04261852} the lift of a contactomorphism from $M$ to $P$ is unique, we conclude that $C_{\exp(tX)} = L_t$.
	
	Finally, for all $t\in\R$ we have
	\[
	e^{t\ad_X} = \Ad_{\exp(tX)} = d C_{\exp(tX)}|_e = d L_t|_e = e^{tD}
	\]
	and thus $D=\ad_X$.
	
	For the last part of the statement, we need to show that $x\mapsto\Ad_x$ is injective on $Q$. 
	So, suppose that $x\in Q$ is such that $\Ad_x$ is the identity map on $P$.
	Since $Q=\exp(\frk q)$, there is $v\in\frk q$ such that $x=\exp(v)$ and thus $\Ad_x=e^{\ad_v}$.
	The vector $v$ can be decomposed as $v=\sum_{j\ge k}v_j$ with $v_j\in\frk g_j$ and $v_k\neq0$, where $k\ge1$.
	If we denote by $\pi_{k-1}$ the projection $\frk p\to\frk g_{k-1}$ given by the grading of $\frk p$, then, for every $w\in\frk g_{-1}$ we have
	\[
	\pi_{k-1}(e^{\ad_v}(w)) = [v,w]
	\]
	Since $e^{\ad_v}(w)=w\in\frk g_{-1}$, then $[v,w]=0$.
	We obtain that $\ad_v|_{\frk g_{-1}}=0$ and thus, by definition of Tanaka prolongation, $v=0$.
	We conclude that $x=e$ and thus $\Ad$ is injective.
\end{proof}

%%%%%%%%%%%%%%%%%%%%%%%%%%%%%%%%%%%%%%%%%%%%%%%%%%%%%%%%%%%%%%%
%%%%%%%%%%%%%%%%%%%%%%%%%%%%%%%%%%%%%%%%%%%%%%%%%%%%%%%%%%%%%%%
\section{Modifications of stratified groups}\label{mod}
A \emph{polarised Lie algebra} is a pair $(\frk s,\frk s_{-1})$ where $\frk s$ is a Lie algebra and $\frk s_{-1}$ is a bracket-generating subspace.
We say that
two polarised Lie algebras $(\frk s,\frk s_{-1})$ and $(\frk s',\frk s_{-1}')$ are \emph{isomorphic} if there is a Lie algebra isomorphism $\phi:\frk s\to\frk s'$ such that $\phi(\frk s_{-1})=\frk s_{-1}'$.
Given a stratified Lie algebra $\frk g$ and a finite dimensional Tanaka prolongation $\frk p={\rm Prol}(\frk g,\frk g_0)$, 
a \emph{modification of $\frk g$ in $\frk p$} is a polarized algebra $(\frk s,\frk s_{-1})$ where $\frk s\subset \frk p$ is a subalgebra such that $\frk p = \frk s\oplus \frk q$ and $\frk s_{-1}=(\frk g_{-1}\oplus\frk q)\cap\frk s$.
In other words, a modification of $\frk g$ in $\frk p$ is a subalgebra $\frk s\subset \frk p$ of the form
\[
\frk s:= \{X+\sigma(X)\,:\, X\in\frk g\},
\]
for some $\sigma:\frk g\to \frk q$ linear, 
endowed with the polarization
\[
\frk s_{-1} := \{X+\sigma(X)\,:\, X\in \frk g_{-1}\} .
\]
Notice that $\frk s_{-1}$ bracket generates $\frk s$.
Indeed, on the one hand, $\frk s$ has the same dimension as $\frk g$.
On the other hand, one can easily check that, for iterated brackets of length $k\ge0$, we have 
\[
\left([\frk s_{-1},\dots[\frk s_{-1},\frk s_{-1}]\dots]\text{ mod} \bigoplus_{j\ge-k}\frk g_j \right)
=\left( [\frk g_{-1},\dots[\frk g_{-1},\frk g_{-1}]\dots]\text{ mod}\bigoplus_{j\ge-k}\frk g_j\right).
\]

If $S$ is the connected Lie subgroup of $P$ with $T_eS=\frk s$, we call the pair $(S,\Delta_S)$ \emph{modification of $G$ in $P$}, where $(\Delta_S)_e=\frk s_{-1}$.
If $(\frk s',\frk s'_{-1})$ is a polarized Lie algebra that is isomorphic to a modification of $\frk g$ in $\frk p$, then we just say that $\frk s$ is \emph{a modification of $\frk g$}.
Similarly, \emph{a modification of $G$} is any polarized group $(S,\Delta_S)$ whose Lie algebra is a modification of $\frk g$.

\begin{lemma}\label{lem04071119}
Let $S$ be a modification of $G$ in $P$.
The restriction $\pi|_S: S\to M$ is a contactomorphism when restricted from a neighbourhood of $e_S$ to one of $\pi(e_S)$.
\end{lemma}
\begin{proof}
We denote by $o$ the base point $\pi(e_S)$ in $M$.
Observe that $d(\pi|_S)_{e_S}: T_{e_S}S=\frk s \to T_o M$ is the restriction to $\frk s$ of $d\pi_{e_P}:\frk p\to T_b M$. Since the kernel of $d\pi_{e_P}$ is $\frk q$, and since $\frk q\cap \frk s=\{0\}$, $d(\pi|_S)_{e_S}$ is injective. Moreover, $\dim \frk s=\dim \frk g= \dim M$, so that $d(\pi|_S)_{e_S}$ is a linear isomorphism. In particular, $\pi|_S$ is a diffeomorphism between two open neighbourhoods of $e_S$ and $o$, respectively. Finally, on the one hand 
\[
d(\pi|_S) (\Delta_S)
= d\pi (\Delta_P\cap TS) 
\subseteq \Delta_M , 
\]
while on the other hand $\dim(\Delta_S)_s=\dim(\frk g_{-1}) = \dim (\Delta_M)_{\pi(s)}$ for all $s\in S$.
So, at all points $s$ where $d(\pi|_S)_s$ is injective we have $d(\pi|_S)_s(\Delta_S)_s=(\Delta_M)_{\pi(s)}$.
\end{proof}

By Lemma~\ref{lem04071119}, both maps
\[
\psi^G_S=\pi|_S^{-1}\circ\pi|_G:U_G\to U_S 
\qquad
\psi^S_G=\pi|_G^{-1}\circ\pi|_S:U_S\to U_G
\]
are contactomorphism between a neighborhood $U_G$ of $e_G$ in $G$ and a neighborhood $U_S$ of $e_S$ in $S$.
They are one the inverse of the other. 
One can also easily prove that the differential $d\psi^G_S|_{e_G}:\frk g\to\frk s$ is the map $X\mapsto X+\sigma (X)$.

Using the maps $\psi^G_S$ and $\psi^S_G$, the following theorem is a direct consequence of Lemma~\ref{lem04071119}.

\begin{theorem}\label{thm:mainone}
	Modifications of a stratified Lie group $G$ are locally contactomorphic to $G$.
\end{theorem}

\begin{remark}\label{rem05101758}
	If we are given $G$ and $S$ in $P$ (for instance as matrix groups), then for any $s\in S$ the image $\psi^S_G(s)$ is the only element $g$ of $G$, if it exists, such that $(sQ)\cap G = \{g\}$.
	Such an element is unique because $\pi:G\to P/Q$ is injective.
\end{remark}

The following theorem is an an inverse of Theorem~\ref{thm:mainone} in the rigid case.

\begin{theorem}\label{teo03291510}
	Suppose that $G$ is a rigid stratified group and that $(S,\Delta_S)$ is a polarized Lie group that is locally contactomorphic to $G$.
	Then $(S,\Delta_S)$ is a modification of $G$.
\end{theorem}
\begin{proof}
	Let $\frk p$ be the full Tanaka prolongation of $\frk g$.
	Let $\psi:U_S\to U_G$ be a contactomorphism from an open subset $U_S\subset S$ to $U_G\subset G$.
	Up to composing $\psi$ with left translations on $S$ and on $G$, we may assume $\psi(e_S)=e_G$.
	
	Let $\frk s^\dagger\subset\Gamma(TS)$ be the Lie algebra of right-invariant vector fields on $S$.
	Since $\frk s^\dagger$ is made of contact vector fields on $S$ and since the Tanaka prolongation of $\frk g$ coincides canonically with the Lie algebra of germs of contact vector fields on $G$,  $\psi_*:\Gamma(TU_S)\to \Gamma(TU_G)$ gives an injective Lie algebra morphism $\psi_*:\frk s^\dagger\into\frk p$.
	
	Notice that if $X\in\frk s^\dagger$ is such that $\psi_*X(e_G)=0$, then $X=0$.
	Therefore, $\psi_*(\frk s^\dagger) \cap\frk q=\{0\}$.
	Since $S$ and $G$ have the same dimension, we obtain that 
	\[
	\psi_*(\frk s^\dagger) = \{X+\sigma X:X\in\frk g\}
	\]
	for some linear map $\sigma:\frk g\to\frk q$.
	
	Finally, since $d\psi((\Delta_S)_{e_S})=(\Delta_G)_{e_G}$, we obtain that 
	\[
	\psi_* \{X\in\frk s^\dagger:X(e_G)\in(\Delta_S)_{e_S}\} = \{X+\sigma X:X\in\frk g_{-1}\} .
	\]
	We conclude that $(\psi_*(\frk s^\dagger), d\psi((\Delta_S)_{e_S}) )$ is a modification of $\frk g$ in $\frk p$.
\end{proof}

\begin{remark}
	In the case $G$ is not rigid, i.e., the full Tanaka prolongation of $\frk g$ is infinite dimensional, then the argument in the proof of Theorem~\ref{teo03291510} does not work.
	However, the example of the Heisenberg group, which is not rigid, shows that it may still be possible to obtain as modifications all Lie groups that are locally contactomorphic to $G$.
	See Section~\ref{heisenberg}.
\end{remark}

In the rigid case, isomorphisms of modifications are all elements of $\Aut(\frk p,\frk g)$:

\begin{theorem}\label{thm05262101}
	Suppose $\frk g$ is rigid.
	If $\frk s,\frk s'$ are two modifications of $\frk g$ in $\frk p$, and if there is an isomorphism $\phi:\frk s\to\frk s'$ such that $\phi(\frk s\cap(\frk g_{-1}\oplus\frk q))=\frk s'\cap(\frk g_{-1}\oplus\frk q)$, 
	then there is a unique $\ell\in\Aut(\frk p,\frk g)$ such that $\phi=\ell|_{\frk s}$.
\end{theorem}
\begin{proof}
	Let $S,S'<P$ be the subgroups of $P$ whose Lie algebra are $\frk s$ and $\frk s'$ respectively, endowed with the polarizations induced by $P$, e.g., $\Delta_{S} = \Delta_P\cap TS$.
	The map $\phi$ defines a local contactomorphism $\Phi:\Omega\to\Phi(\Omega)$, $\Omega\subset S$ open with $e\in\Omega$.
	We may assume that $\pi|_\Omega:\Omega\to \pi(\Omega)\subset M$ and $\pi|_{\Phi\Omega}:\Phi(\Omega)\to \pi(\Phi\Omega)\subset M$ are contactomorphisms, see Lemma~\ref{lem04071119}.
	Define $U:=\pi(\Omega)$ and $f:=\pi\circ\Phi\circ\pi|_\Omega^{-1}:U\to f(U)=\pi(\Phi\Omega)$.
	The map $f$ is then a contactomorphism.
	By Theorem~\ref{thm5e905026}, there is a Lie group automorphism $L:P\to P$ such that $f(\pi(p))=\pi(L(p))$ for all $p\in\pi^{-1}(U)$.
	Now, we claim that the map $L_*:\frk p\to \frk p$ restricted to $\frk s$ is equal to $\phi$.
	Indeed, if $X\in\frk s$, then 
	\begin{align*}
	L_*[X]
	&= \left. L_*[X^\dagger] \right|_{e} 
		= \left. (\pi_*|_{\frk p^\dagger})^{-1}\circ\pi_*\circ L_*\circ(\pi_*|_{\frk p^\dagger})^{-1}\circ\pi_* [X^\dagger] \right|_{e} \\
	&= \left. (\pi_*|_{\frk p^\dagger})^{-1}\circ f_*\circ\pi_* [X^\dagger] \right|_{e} 
		= \left. (\pi_*|_{\frk p^\dagger})^{-1}\circ (\pi\circ\Phi\circ\pi|_\Omega^{-1})_*\circ\pi_* [X^\dagger] \right|_{e} \\
	&= \left. \Phi_* [X^\dagger|_S] \right|_{e} 
		= \phi[X] .\qedhere
	\end{align*}
\end{proof}

%%%%%%%%%%%%%%%%%%%%%%%%%%%%%%%%%%%%%%%%%%%%%%%%%%%%%%%%%%%%%%%
%%%%%%%%%%%%%%%%%%%%%%%%%%%%%%%%%%%%%%%%%%%%%%%%%%%%%%%%%%%%%%%
\section{Examples}\label{exsection}

In this section we consider a few applications of our main results. First, we observe that  every  left-invariant three dimensional contact structure in $\R^3$ is a modification of the Heisenberg group, and consequently  is locally contactomorphic to it. Although this is of course a consequence of the more general Darboux Theorem, we believe it is a good example for presenting our technics. Second, we study some modifications of the free nilpotent Lie algebra $\frk f_{24}$. In this case we are able to find a nilpotent modification $(N,\Delta_N)$ of the  stratified group $F_{24}$ corresponding to $\frk f_{24}$ that is globally contactomorphic to $F_{24}$, but not isomorphic. In particular, if we endow  $N$ and $F_{24}$ with left-invariant sub-Riemannian distances, our example shows two nilpotent Lie groups that are bi-Lipschitz on every compact set but not isomorphic.

\subsection{Modifications of the Heisenberg group}\label{heisenberg} 
We study the consequences of the results of the previous section in the case where $\frk g$ is the three-dimensional Heisenberg algebra. 
It is well known that the full Tanaka prolongation of the Heisenberg Lie algebra $\frk g$ is infinite. 
However, there is a number a different choices of subalgebras $\frk g_0\subset {\rm Der}(\frk g)$ that generate finite dimensional prolongations.
We shall show that these finite prolongations are enough to recover all polarized Lie groups that are locally contactomorphic with the Heisenberg group, i.e., all three dimensional Lie groups with a non-trivial polarization:

\begin{theorem}\label{thm5ea01a1d}
	Let $(\frk s,\frk s_{-1})$ be a three dimensional polarised Lie algebra such that $\dim(\frk s_{-1})=2$.
	Then there is a finite-dimensional prolongation $\frk p$ of the Heisenberg Lie algebra $\frk h$ so that $(\frk s,\frk s_{-1})$ is isomorphic to a modification in $\frk p$.
\end{theorem}
 
Our study is based on a classification of three-dimensional Lie algebras due to several authors.
We summarise the results we need in the following theorem.

\begin{proposition}\label{prop04111840}
	Let  $(\frk s,\frk s_{-1})$ be a three-dimensional polarised Lie algebra such that $\dim(\frk s_{-1})=2$.
	Then there is a basis $(f_1, f_2, f_3)$ of $\frk s$ with $\frk s_{-1}=\Span\{f_1,f_2\}$ such that 
	exactly one of the following cases occurs:
	\begin{enumerate}[label=(\Alph*)]
	\item\label{caseA}
	$[f_1, f_2]=f_3$, 
	$[f_1, f_3]=\alpha f_2 +\beta f_3$ and 
	$[f_2, f_3]=0$,
	for some $\alpha\in \R $ and $\beta\in\{0,1\}$. 
	In this case $\frk s$ is solvable and the non-isomorphic
	cases are exactly the following four:
	\begin{enumerate}[label=(A.\arabic*)]
	\item\label{caseA1}
	$[f_1, f_2]=f_3$, 
	$[f_1, f_3]=0$ and 
	$[f_2, f_3]=0$;
	\item\label{caseA2}
	$[f_1, f_2]=f_3$, 
	$[f_1, f_3]=f_2$ and 
	$[f_2, f_3]=0$;
	\item\label{caseA3}
	$[f_1, f_2]=f_3$, 
	$[f_1, f_3]=-f_2$ and 
	$[f_2, f_3]=0$;
	\item\label{caseA4}
	$[f_1, f_2]=f_3$, 
	$[f_1, f_3]=\alpha f_2 + f_3$ and 
	$[f_2, f_3]=0$.
	\end{enumerate}
	\item\label{caseB}
	$[f_1, f_2]=f_3$,
	$[f_1, f_3]=- f_2$,
	$[f_2, f_3]=f_1$.
	In this case $\frk s= \frk{su}(2)$ is simple.
	\item\label{caseC}
	$[f_1, f_2]=f_3$,
	$[f_1, f_3]=- f_1$,
	$[f_2, f_3]=f_2$.
	In this case $\frk s= \frk{sl}(2,\R)$ is simple.
	\item\label{caseD}
	$[f_1, f_2]=f_3$,
	$[f_1, f_3]=f_2$,
	$[f_2, f_3]=-f_1$.
	In this case $\frk s= \frk{sl}(2,\R)$ is simple.
	\end{enumerate}
\end{proposition}
Part of the proof of Proposition~\ref{prop04111840} is based on the following lemma, see \cite{MR3365789}.
\begin{lemma}[Baudoin--Cecil]\label{BC}
Let $S$ be a three-dimensional solvable Lie group endowed with a left-invariant sub-Riemannian structure $(\Delta_S, g)$. 
There exist vectors $e_1, e_2, e_3$ linearly independent in $\frk s$, $\alpha\in \R $ and $\beta\geq 0$ such that
$e_1, e_2$ is an orthonormal basis of $(\Delta_S)_e$ and
\begin{equation}\label{bracketsSolv}
[e_1, e_2]=e_3, \quad
[e_1, e_3]=\alpha e_2 +\beta e_3,\quad
[e_2, e_3]=0.
\end{equation}
\end{lemma}
\begin{proof}[Proof of Proposition~\ref{prop04111840}]
	If $\frk s$ is a three dimensional Lie algebra, then it is either solvable or simple. Indeed,  the claim follows from the Levi decomposition and the fact that there are no simple Lie groups of dimension $1$ or $2$.
	If $\frk s$ is simple, then the $(\frk s,\frk s_{-1})$ falls into one the cases~\ref{caseB},~\ref{caseC} or~\ref{caseD}, see~\cite{MR2902707}.
	Notice that the cases~\ref{caseC} and~\ref{caseD} are not isomorphic as polarised Lie algebras because $\ad_{f_3}$ is a reflexion of $\frk s_{-1}$ in case~\ref{caseC}, while in case~\ref{caseD} it is a rotation.
	
	If $\frk s$ is solvable, then we apply Lemma~\ref{BC} and obtain case~\ref{caseA}.
	However, since the classification in Lemma~\ref{BC} is up to isometry, we have to further discriminate to obtain non-isomorphic subcases.
	So, if $f_1,f_2,f_3$ is a basis of $\frk s$ with $\frk s_{-1}=\Span\{ f_1,f_2\}$, $[f_1,f_2]=f_3$ and $[f_2,f_3]=0$, then we must have
	\[
	\begin{cases}
	f_1 &= a_1^1 e_1 + a_1^2 e_2 \\
	f_2 &= a_2^2 e_2 \\
	f_3 &= a_1^1a_2^2 e_3 ,
	\end{cases}
	\]
	for some real coefficients.
	The third bracket relation is
	\[
	[f_1,f_3] = \alpha (a_1^1)^2 f_2 + \beta a_1^1 f_3 .
	\]
%	The bracket products are
%	\begin{align*}
%	[f_1,f_2] &= f_3 \\
%	[f_1,f_3] &= \alpha (a_1^1)^2 f_2 + \beta a_1^1 f_3\\
%	[f_2,f_3] &= 0 .
%	\end{align*}
	Since $\alpha\in\R$ and $\beta\ge0$, in each case we can choose $a_i^j$ in the following way:
	\begin{align*}
	\alpha=\beta=0 
		&& a_1^1=1,\ a_1^2=0,\ a_2^2=1 : 
		&& [f_1,f_3] &= 0 \\
	\beta>0,\ \alpha\in\R 
		&& a_1^1=\frac1\beta,\ a_1^2=0,\ a_2^2=1 :
		&& [f_1,f_3] &= \frac{\alpha}{\beta^2} f_2 + f_3 \\
	\beta=0,\ \alpha>0 
		&& a_1^1=\frac1{\sqrt\alpha},\ a_1^2=0,\ a_2^2=1 :
		&& [f_1,f_3] &=  f_2  \\
	\beta=0,\ \alpha<0 
		&& a_1^1=\frac1{\sqrt{|\alpha|}},\ a_1^2=0,\ a_2^2=1 :
		&& [f_1,f_3] &= - f_2 
	\end{align*}
	
	Now, we want to show that cases \ref{caseA1}, \ref{caseA2}, \ref{caseA3} and \ref{caseA4} are not isomorphic to each other.
	Notice that $\ell:=\Span\{f_3\}=[\frk s_{-1},\frk s_{-1}]$ and $\frk s^{(2)}:=[\frk s,\frk s]$ are invariant under isomorphisms of polarised Lie algebras.
	
	First, case~\ref{caseA1} is not isomorphic to the others because in case~\ref{caseA1} we have $\frk s^{(2)}=\Span\{f_3\}$ while in all other three cases we have $\frk s^{(2)}=\Span\{f_2,f_3\}$.
	
	Second, case~\ref{caseA4} is not isomorphic to the others because in case~\ref{caseA4} we have $[\ell,\frk s_{-1}] \not\subset\frk s_{-1}$ while in all other cases we have $[\ell,\frk s_{-1}] \subset\frk s_{-1}$.
	
	Third, for different choices of $\alpha\in\R$ in case~\ref{caseA4} we get non-isomorphic polarised Lie algebras:  To prove this, we shall show that the parameter $\alpha$ is independent of the choice of the basis.
	So, suppose that $g_1,g_2,g_3\in\frk s$ form another basis with $\frk s_{-1}=\Span\{g_1,g_2\}$, $[g_1,g_2]=g_3$, $[g_2,g_3]=0$ and $[g_1,g_3]=\alpha' g_2+g_3$.
	Then one easily shows that $g_1=xf_1+y f_2$, $g_2=\mu f_2$ and $g_3=\lambda f_3$, for some $x,y,\lambda,\mu\in\R$ with $\frac{x\mu}{\lambda}=1$.
	Moreover, $[g_1,g_3] = \alpha \frac{x\mu}{\lambda} g_2 + x g_3$, which implies $x=1$ and $\alpha=\alpha'$.
	
	Finally, cases~\ref{caseA2} and~\ref{caseA3} are not isomorphic to each other, because in case~\ref{caseA2} it holds $\ad_{f_1}|_{\frk s^{(2)}}^2=\Id|_{\frk s^{(2)}}$, while while in case~\ref{caseA3} it holds $\ad_{f_1}|_{\frk s^{(2)}}^2=-\Id|_{\frk s^{(2)}}$.
\end{proof}

\begin{proof}[Proof of Theorem~\ref{thm5ea01a1d}]
	Let us fix the notation for the Heisenberg Lie algebra. Fix  a basis $e_1,e_2,e_3$ so that $[e_1,e_2]=e_3$, and choose $\frk g_{-1}={\rm span}\{e_1,e_2\}$. 
	The space ${\rm Der}(\frk g)$ of the strata preserving derivations of $\frk g$ may be identified with $\frk{gl}(2,\R)$.  
	
	First, we  consider 
	\[
	\frk g_0:=\{D\in {\rm Der}(\frk g)\,:\, D(e_1)\subseteq \R e_1 \text{ and }D(e_2)\subseteq \R e_2\}.
	\]
	In this case, ${\rm Prol}(\frk g,\frk g_0) = \frk{sl}(3,\R)=\frk g \oplus \frk q$ (see, e.g., \cite{MR1984102}), where $\frk g$ is identified with the Lie algebra generated by
	\begin{equation}\label{eq05101826}
	e_1=\begin{pmatrix}
	0&1&0\\
	0&0&0\\
	0&0&0
	\end{pmatrix}, \quad e_2=\begin{pmatrix}
	0&0&0\\
	0&0&1\\
	0&0&0
	\end{pmatrix},\quad e_3= \begin{pmatrix}
	0&0&1\\
	0&0&0\\
	0&0&0
	\end{pmatrix}.
	\end{equation}
	and $\frk q$ is the set of matrices in $\frk{sl}(3,\R)$ of the form
	\[
	\begin{pmatrix}
	*&0&0\\
	*&*&0\\
	*&*&*
	\end{pmatrix}
	\]
	The modifications of $\frk g$ in $\frk{sl}(3,\R)$ are the subalgebras of $\frk{sl}(3,\R)$ of the form $\{X+\sigma(X)\,:\,X\in \frk g\}$,
	for some linear map $\sigma: \frk g \to \frk q$. 
	We  show that all three dimensional Lie algebras with a bracket generating plane are graphs of such a $\sigma$: 
	
	Case~\ref{caseA}:
	If $\frk s$ is solvable, then define $\sigma$ by the assignments:
	\[
	\sigma(e_1)=\begin{pmatrix}
	\frac{2\beta}{3} &0&0\\
	\alpha&-\frac{\beta}{3}&0\\
	0&0&-\frac{\beta}{3}
	\end{pmatrix},\quad \sigma(e_2)=\sigma(e_3)=0.
	\]
	It is easy to check that vectors $f_i:=e_i+\sigma(e_i)$, $i=1,2,3$, satisfy the bracket relations of case~\ref{caseA} in Proposition~\ref{prop04111840}.
	
	Case~\ref{caseB}:
	For this case, we choose
	\[
	\sigma(e_1)=\begin{pmatrix}
	0 &0&0\\
	-1&0&0\\
	0&0&0
	\end{pmatrix},\quad\sigma(e_2)=\begin{pmatrix}
	 0&0&0\\
	0&0&0\\
	0&-1&0
	\end{pmatrix},\quad\sigma(e_3)=\begin{pmatrix}
	0&0&0\\
	0&0&0\\
	-1&0&0
	\end{pmatrix}.
	\]
	
	Case~\ref{caseC}:
we obtain the brackets in ~\ref{caseC} by choosing
	\[
	\sigma(e_1)=0,\quad\sigma(e_2)=\begin{pmatrix}
	 0&0&0\\
	1/2&0&0\\
	0&0&0
	\end{pmatrix},\quad\sigma(e_3)=\begin{pmatrix}
	1/2&0&0\\
	0&-1/2&0\\
	0&0&0
	\end{pmatrix}.
	\]
	
	Case~\ref{caseD}:
	In this case we use the finite prolongation $\frk{su}(2,1)$ of the Heisenberg algebra, as in~\cite[p313]{MR788413}.
	Let 
	\[
	J= \begin{pmatrix}1&0&0\\0&-1&0\\0&0&-1\end{pmatrix} .
	\]
	The Lie algebra $\frk{su}(2,1)$ is given by $3\times3$ complex matrices $A$ with zero trace and such that $A^*J+JA=0$, where $A^*$ is the hermitian transpose of $A$.
	Define the Lie algebra automorphism $\theta:\frk{su}(2,1)\to\frk{su}(2,1)$, $\theta A:=JAJ$.
	Define
	\begin{align*}
	X  & = \begin{pmatrix} 0 & i & 0 \\ -i & 0 & -i \\ 0 & -i & 0 \end{pmatrix}  & 
	Y  & = \begin{pmatrix} 0 & 1 & 0 \\ 1 & 0 & 1 \\ 0 & -1 & 0 \end{pmatrix}  & 
	Z  & = \begin{pmatrix} 2i & 0 & 2i \\ 0 & 0 & 0 \\ -2i & 0 & -2i \end{pmatrix} \\
	H  & = \begin{pmatrix} 0 & 0 & 1 \\ 0 & 0 & 0 \\ 1 & 0 & 0 \end{pmatrix}  &&& 
	U  & = \begin{pmatrix} i & 0 & 0 \\ 0 & -2i & 0 \\ 0 & 0 & i \end{pmatrix} \\
	\theta X  & = \begin{pmatrix} 0 & -i & 0 \\ i & 0 & -i \\ 0 & -i & 0 \end{pmatrix}  & 
	\theta Y  & = \begin{pmatrix} 0 & -1 & 0 \\ -1 & 0 & 1 \\ 0 & -1 & 0 \end{pmatrix}  & 
	\theta Z  & = \begin{pmatrix} 2i & 0 & -2i \\ 0 & 0 & 0 \\ 2i & 0 & -2i \end{pmatrix}
	\end{align*}
	The grading of $\frk{su}(2,1)$ is
	\begin{align*}
	\frk g_{-2}(\frk g) &= \Span\{Z\} \\
	\frk g_{-1}(\frk g) &= \Span\{X,Y\} \\
	\frk g_{0}(\frk g) &= \Span\{H,U\} \\
	\frk g_{1}(\frk g) &= \Span\{\theta X,\theta Y\} \\
	\frk g_{2}(\frk g) &= \Span\{\theta Z\},
	\end{align*}
	where $\frk g_{-2}(\frk g)\oplus \frk g_{-2}(\frk g) = \frk g$ is the Heisenberg Lie algebra: notice that $[X,Y]=Z$ while $[X,Z]=[Y,Z]=0$. 
	So, $\frk q = \Span\{H,U,\theta X,\theta Y,\theta Z\}$.
	Define $\sigma:\frk g\to\frk q$ by setting
	\begin{align*}
	\sigma X &:= -\frac1{16} \theta X + i\frac{9}{16} \theta Y 
		= \begin{pmatrix} 0 & -i\frac12 & 0 \\ -i\frac58 & 0 & i\frac58 \\ 0 & -i\frac12 & 0 \end{pmatrix} , \\
	\sigma Y &:= - i\frac{9}{16} \theta X - \frac1{16} \theta Y
		= \begin{pmatrix} 0 & -\frac12 & 0 \\ \frac58 & 0 & -\frac58 \\ 0 & -\frac12 & 0 \end{pmatrix} , \\
	\sigma Z &:= -i \frac{9}{4} H  + \frac14 U  - \frac{5}{16} \theta Z
		= \begin{pmatrix} -i\frac38 & 0 & -i\frac{13}8 \\ 0 & -i\frac12 & 0 \\ -i\frac{23}8 & 0 & i\frac78 \end{pmatrix} .
	\end{align*}
	One can easily check that $f_1=X+\sigma X$, $f_2=Y+\sigma Y$ and $f_3=Z+\sigma Z$ form a basis of a Lie subalgebra of $\frk{su}(2,1)$ satisfying the relations of Case~\ref{caseD}.
\end{proof}
\begin{remark}
	The map $\sigma$ above can easily be found using the software Maple and it is not the unique.
\end{remark}

\subsubsection{Rigid motions of the plane as a modification of the Heisenberg group}
We conclude this section discussing more in detail the case of the group of rigid motions of the plane as a modification of the Heisenberg group.
At a group level, we may represent points in the Heisenberg group $\bb H$ as matrices in $SL(3,\R)$ by
\[
H(x_1,x_2,x_3):=\begin{pmatrix} 1&x_1&x_3\\
0&1&x_2\\
0&0&1
\end{pmatrix},
\]
for $x_1,x_2,x_3\in\R$.

The Lie algebra of the the group of rigid motions of the plane $E(2)$ corresponds to the case~\ref{caseA} with $\alpha=-1$ and $\beta=0$.
The corresponding representation in $\frk{sl}(3,\R)$ given in the previous theorem is the span of the vectors
\[
f_1 = \begin{pmatrix}0&1&0\\-1&0&0\\0&0&0\end{pmatrix} ,
\qquad
f_2 = \begin{pmatrix}0&0&0\\0&0&1\\0&0&0\end{pmatrix} ,
\qquad
f_3 = \begin{pmatrix}0&0&1\\0&0&0\\0&0&0\end{pmatrix} .
\]
At the group level, the points of $E(2)$ inside $SL(3,\R)$ are parametrized by
\[
R(y_1,y_2,y_3):=\begin{pmatrix} \cos y_1&\sin y_1&y_3\\
-\sin y_1&\cos y_1&y_2\\
0&0&1
\end{pmatrix}
\]
where $y_1\in\R/(2\pi\Z)$ and $y_2,y_3\in\R$.

With the procedure described in Remark~\ref{rem05101758}, we find the mapping $E(2)\to \bb H$:
\[
R(y_1,y_2,y_3) \mapsto H(\tan y_1,y_2,y_3),
\]
which is defined on the domain $(-\pi/2,\pi/2)\times \R^2$.

%%%%%%%%%%%%%%%%%%%%%%%%%%%%%%%%%%%%%%%%%%%%%%%%%%%%%%%%%%%%%%%
\subsection{Modifications of the free nilpotent Lie group $F_{24}$}\label{freemod}
We consider with the free nilpotent Lie algebra $\frk{f}_{24}={\rm span}\{e_i\,:\,i=1,\dots,8\}$ of rank 2 and step 4 and the corresponding simply connected Lie group $F_{2,4}$. 
We will prove the following result
\begin{theorem}\label{thm5ea02e50}
	There exists a nilpotent Lie group $S$, not isomorphic to $F_{2,4}$, that is a modification of $F_{2,4}$ and is globally contactomorphic to $F_{2,4}$.
\end{theorem}
\begin{proof}
The Lie brackets in $\frk{f}_{24}$ are
\begin{align*}
&[e_2,e_1]=e_3,\quad [e_3,e_1]=e_4,\quad [e_3,e_2]=e_5,\\
& [e_4,e_1]=e_6,\quad [e_5,e_1]=e_7,,\quad [e_4,e_2]=e_7,\quad [e_5,e_2]=e_8.
\end{align*}
It is known that the full Tanaka prolongation of $\frk{f}_{24}$ is $\frk p=\frk{f}_{24} \oplus {\rm Der}(\frk g)$, with ${\rm Der}(\frk g)\simeq \frk{gl}(2,\R)$ (see \cite{MR2365778}). Therefore, the modifications of $\frk{f}_{24}$ are subalgebras of $\frk p$ that are graphs of some linear map
$\sigma: \frk{f}_{24}\to \frk{gl}(2,\R)$. 
Here we only consider $\sigma$ that on the basis of $\frk{f}_{24}$ is zero except for $\sigma(e_1)$. Imposing that the graph is a Lie algebra, a direct computation shows that
\[
\sigma(e_1)=\begin{pmatrix} a&0\\
c&b\\
\end{pmatrix},
\]
where $a,b,c\in \R$.
We obtain a three parameter family $\frk{s}(a,b,c)$ of Lie algebras with basis 
$f_1,\dots,f_8$, where $f_1=e_1+\sigma(e_1)$ and $f_i=e_i$ for $i=2,\dots,8$,
and brackets
\begin{align*}
	& 		[f_2, f_1] = f_3- bf_2, 
	\quad 	[f_3, f_1] = f_4-(a+b)f_3, 
	\quad 	[f_3, f_2] = f_5, \\
	&  		[f_4, f_1] = f_6 - c f_5 - (2a+b)f_4 , 
	\quad 	[f_4, f_2] = f_7,
	\quad 	[f_5, f_1] = f_7 - (a+2b)f_5 ,\\
	&  		[f_1, f_6] = 2cf_7 + (3a+b)f_6, 
	\quad 	[f_1, f_7] = cf_8 + 2(a+b)f_7, \\
	& 		[f_1, f_8] = (a+3b)f_8  , 
	\quad 	[f_5, f_2] = f_8.
\end{align*}

In particular, setting $a=b=0$ gives a one parameter family of nilpotent Lie algebras $\frk{s}(c)$. We now find the contact mapping $\Psi$ from $S(c)$ to $F_{24}$ when $c=1$, as in Remark~\ref{rem05101758}.
Every point in $S(1)$ is of the form $\exp_P (\sum x_i f_i)$. 
Following \cite{MR1958588}, $\exp_P (\sum x_i f_i)= (\mathcal{E}_{F_{24}}(x_1\sigma(e_1);\sum x_i e_i),\exp_{GL}(x_1\sigma(e_1))\in F_{24}\rtimes GL(2,\R)$, where 
$\mathcal{E}_{F_{24}}(x_1e_1;\sum x_i e_i)=\gamma(1)$ and $\gamma:[0,1]\to F_{24}$ is the solution of
\[
\begin{cases}
\gamma'(t)&=dL_{\gamma(t)} \exp_{GL}(t x_1\sigma(e_1))(\sum x_i e_i)\\
\gamma(0)&=e_{F_{24}}.
\end{cases}
\]
The image of this point via $\Psi$ is going to be that element $p\in F_{24}$ such that $gQ=\exp_P (\sum x_i f_i)Q$, i.e.,
\[
\Psi\left(\exp_P (\sum x_i f_i)\right) = \mathcal{E}_{F_{24}}(x_1e_1;\sum x_i e_i) .
\]
To compute this, first,
\begin{align*}
v &:= \exp_{GL}(t x_1\sigma(e_1))\left(\sum x_i e_i\right) \\
	&=\left(x_1,x_1^2t+x_2,x_3,x_4,x_5+tx_1x_4,x_6,x_7+2tx_1x_6,x_8+tx_1x_7+t^2x_1^2x_6\right).
\end{align*}
	%%%%%%%%%%%%%%%%%%%%%%%%%%%%%%%%%%%%%%%%%%%%%%%%%%%%%%%%%%%%%%%
	% Computation of the exponential with SymPy:
	%#!/usr/bin/env python
	%# coding: utf-8
	%# In[1]:
	%from sympy import *
	%# In[2]:
	%init_printing()
	%# In[3]:
	%a = symbols('a');
	%b = symbols('b');
	%c = symbols('c');
	%t = symbols('t');
	%x1 = symbols('x1');
	%# In[4]:
	%sigma=Matrix([ 
	%[ a , 0 , 0   , 0    , 0    , 0    , 0      , 0    ], 
	%[ c , b , 0   , 0    , 0    , 0    , 0      , 0    ], 
	%[ 0 , 0 , a+b , 0    , 0    , 0    , 0      , 0    ], 
	%[ 0 , 0 , 0   , 2*a+b , 0    , 0    , 0      , 0    ], 
	%[ 0 , 0 , 0   , c    , a+2*b , 0    , 0      , 0    ], 
	%[ 0 , 0 , 0   , 0    , 0    , 3*a+b , 0      , 0    ], 
	%[ 0 , 0 , 0   , 0    , 0    , 2*c   , 2*(a+b) , 0    ], 
	%[ 0 , 0 , 0   , 0    , 0    , 0    , c      , a+3*b ]]) 
	%# In[5]:
	%exp(t*x1*sigma)
	%# In[6]:
	%sigma0 = sigma.subs(a,0).subs(b,0).subs(c,1) ;
	%sigma0
	%# In[7]:
	%exp(t*x1*sigma0)
	%%%%%%%%%%%%%%%%%%%%%%%%%%%%%%%%%%%%%%%%%%%%%%%%%%%%%%%%%%%%%%%

Second, we need to compute $dL_{\gamma}v$ using the Baker--Campbell--Hausdorff formula:
\[
dL_{\gamma}v 
	= \left.\frac{d}{d h}\right|_{h=0} \exp^{-1}(\exp(\gamma)\exp(hv))
	= v + \frac12 [\gamma,v] + \frac1{12} [\gamma,[\gamma,v]] .
\]
The system of differential equations $\dot\gamma=dL_{\gamma}v$ that we obtain is
\begin{align*}
\dot\gamma_1 &= x_1 \\
\dot\gamma_2 &= t x_1^{2} + x_2 \\
\dot\gamma_3 &= -\frac{1}{2}  t x_1^{2} \gamma_1 - \frac{1}{2}  x_2 \gamma_1 + \frac{1}{2}  x_1 \gamma_2 + x_3 \\
\dot\gamma_4 &= \frac{1}{12}  t x_1^{2} \gamma_1^{2} + \frac{1}{12}  x_2 \gamma_1^{2} - \frac{1}{12}  {\left(\gamma_1 \gamma_2 - 6  \gamma_3\right)} x_1 - \frac{1}{2}  x_3 \gamma_1 + x_4 \\
\dot\gamma_5 &= \frac{1}{12}  x_2 \gamma_1 \gamma_2 - \frac{1}{12}  x_1 \gamma_2^{2} + \frac{1}{12}  {\left(x_1^{2} \gamma_1 \gamma_2 + 6  x_1^{2} \gamma_3 + 12  x_1 x_4\right)} t %\\
	- \frac{1}{2}  x_3 \gamma_2 \\
	&\qquad
	+ \frac{1}{2}  x_2 \gamma_3 
	+ x_5 \\
\dot\gamma_6 &= \frac{1}{12}  x_3 \gamma_1^{2} - \frac{1}{12}  {\left(\gamma_1 \gamma_3 - 6  \gamma_4\right)} x_1 - \frac{1}{2}  x_4 \gamma_1 + x_6 \\
\dot\gamma_7 &= \frac{1}{6}  x_3 \gamma_1 \gamma_2 
	- \frac{1}{12}  x_2 \gamma_1 \gamma_3 
	- \frac{1}{12}  \left(x_1^{2} \gamma_1 \gamma_3  
		+ 6  x_1 x_4 \gamma_1 
		- 6  x_1^{2} \gamma_4 
		- 24  x_1 x_6\right) t 
	\\&\qquad 
	- \frac{1}{12}  \left(\gamma_2 \gamma_3 
		- 6  \gamma_5\right) x_1 
	- \frac{1}{2}  x_5 \gamma_1 
	- \frac{1}{2}  x_4 \gamma_2 
	+ \frac{1}{2}  x_2 \gamma_4 
	+ x_7 \\
\dot\gamma_8 &= t^{2} x_1^{2} x_6 
	+ \frac{1}{12}  x_3 \gamma_2^{2} 
	- \frac{1}{12}  x_2 \gamma_2 \gamma_3 
	- \frac{1}{12}  \left(x_1^{2} \gamma_2 \gamma_3 
		+ 6  x_1 x_4 \gamma_2 
		- 6  x_1^{2} \gamma_5 
		- 12  x_1 x_7\right) t 
	\\&\qquad 
	- \frac{1}{2}  x_5 \gamma_2 
	+ \frac{1}{2}  x_2 \gamma_5 
	+ x_8 .
\end{align*}

Third, we need to integrate this system of ODEs 
with initial conditions $\gamma_i(0)=0$ for every $i=1,\dots,8$. The solution is
\begin{align*}       	
\gamma_{1}(t) &= t x_{1} \\
\gamma_{2}(t) &= \frac{1}{2}  t^{2} x_{1}^{2} + t x_{2} \\
\gamma_{3}(t) &= -\frac{1}{12}  t^{3} x_{1}^{3} + t x_{3} \\
\gamma_{4}(t) &= t x_{4} \\
\gamma_{5}(t) &= -\frac{1}{240}  t^{5} x_{1}^{5} 
	+ \frac{1}{12}  t^{3} x_{1}^{2} x_{3} 
	+ \frac{1}{2}  t^{2} x_{1} x_{4} 
	+ t x_{5} \\
\gamma_{6}(t) &= \frac{1}{720}  t^{5} x_{1}^{5} + t x_{6} \\
\gamma_{7}(t) &= \frac{1}{720}  t^{6} x_{1}^{6} 
	+ \frac{1}{360}  t^{5} x_{1}^{4} x_{2} 
	+ t^{2} x_{1} x_{6} 
	+ t x_{7} \\
\gamma_{8}(t) &= \frac{1}{5040}  t^{7} x_{1}^{7} 
	+ \frac{1}{720}  t^{6} x_{1}^{5} x_{2} 
	+ \frac{1}{720}  \left(x_{1}^{3} x_{2}^{2} 
		+ 3  x_{1}^{4} x_{3}\right) t^{5} 
	\\&\qquad 
	- \frac{1}{12}  \left(x_{1} x_{2} x_{4} 
		- x_{1}^{2} x_{5} 
		- 4  x_{1}^{2} x_{6}\right) t^{3} 
	+ \frac{1}{2}  t^{2} x_{1} x_{7} 
	+ t x_{8} .
\end{align*}
Therefore, the mapping from $S(1)$ to $G$ is $\Psi:\exp_P (\sum x_i f_i)\mapsto\gamma(1)$,
which is a global, surjective smooth contactomorphism.

Finally, $S(1)$ is not isomorphic to $F_{2,4}$ because $S(1)$ has nilpotency step 5 instead of 4, as one can easily see from the expression of the Lie brackets in $\frk s(1)$.
\end{proof}
\begin{remark}
	The mapping from $S(1)$ to $G$ described above is in particular biLipschitz on every compact set,  when the groups are endowed with left-invariant sub-Riemannian distances. Notice, however, that this is not a global quasiconformal mapping.
\end{remark}

	%sol1 := dsolve([diff(x1(t),t) = x_1, diff({\gamma_2(t)},t) = x_1^2*t+x_2, diff({\gamma_3(t)},t) = x_3+1/2*{\gamma_2(t)}*x_1-1/2*x1(t)*(x_1^2*t+x_2), diff({\gamma_4(t)},t) = x_4+1/2*{\gamma_3(t)}*x_1-1/2*x1(t)*x_3+1/12*x1(t)^2*(x_1^2*t+x_2)-1/12*x1(t)*{\gamma_2(t)}*x_1, diff({\gamma_5(t)},t) = a5+t*x_1*x_4+1/2*{\gamma_3(t)}*(x_1^2*t+x_2)-1/2*{\gamma_2(t)}*x_3+1/12*x1(t)*{\gamma_2(t)}*(x_1^2*t+x_2)-1/12*{\gamma_2(t)}^2*x_1, diff({\gamma_6(t)},t) = a6-1/2*x1(t)*x_4+1/2*{\gamma_4(t)}*x_1+1/12*x1(t)^2*x_3-1/12*x1(t)*{\gamma_3(t)}*x_1, diff({\gamma_7(t)},t) = a7+2*x_1*a6-1/2*x1(t)*(x_1*x_4*t+a5)-1/2*{\gamma_2(t)}*x_4+1/2*{\gamma_4(t)}*(x_1^2*t+x_2)+1/2*{\gamma_5(t)}*x_1+1/6*x1(t)*{\gamma_2(t)}*x_3-1/12*x1(t)*{\gamma_3(t)}*(x_1^2*t+x_2)-1/12*{\gamma_2(t)}*{\gamma_3(t)}*x_1, diff({\gamma_8(t)},t) = a8+t*x_1*a7+t^2*x_1^2*a6-1/2*{\gamma_2(t)}*(x_1*x_4*t+a5)+1/2*{\gamma_5(t)}*(x_1^2*t+x_2)-1/12*{\gamma_2(t)}^2*x_3-1/12*{\gamma_2(t)}*{\gamma_3(t)}*(x_1^2*t+x_2), x1(0) = 0, x2(0) = 0, x3(0) = 0, x4(0) = 0, x5(0) = 0, x6(0) = 0, x7(0) = 0, x8(0) = 0]);
%

%%%%%%%%%%%%%%%%%%%%%%%%%%%%%%%%%%%%%%%%%%%%%%%%%%%%%%%%%%%%%%%
%%%%%%%%%%%%%%%%%%%%%%%%%%%%%%%%%%%%%%%%%%%%%%%%%%%%%%%%%%%%%%%
\subsection{Modifications of ultra-rigid stratified groups}\label{ultra}
A stratified Lie algebra $\frk g$ is called \emph{ultra-rigid} if the only automorphisms of $\frk g$ preserving the stratifications are dilations, see \cite{MR3331166}.
In particular, the full Tanaka prolongation of such $\frk g$ is $\frk p=\frk g\rtimes\R$, as semi-direct product of Lie algebras.
In this section we describe all modifications in $\frk g\rtimes\R$ and their equivalence relation.
Many results do not need the assumption of $\frk g$ being ultra-rigid, so we assume this hypothesis only when needed.

Let $\frk g=\bigoplus_{j=-s}^{-1}\frk g_{j}$ be a stratified Lie algebra.
Let $D:\frk g\to\frk g$ be the linear map with $Dv=jv$ for $v\in\frk g_{-j}$.
Notice that $D$ is a derivation of $\frk g$ that preserves the layers and that $\delta_{e^t}=e^{tD}:\frk g\to\frk g$ are the dilations.

The semi-direct product $\frk p:=\frk g\rtimes\R$ is the Lie algebra whose Lie brackets are
	%\footnote{As a reference, see \url{https://en.wikipedia.org/wiki/Lie_algebra_extension\#By_semidirect_sum}}
\[
[(0,a),(Y,0)] = (aDY,0)
\quad\text{hence}\quad
[(X,a),(Y,b)] = ([X,Y]+aDY-bDX,0) .
\]

\begin{proposition}\label{prop05271430}
	Let $\sigma:\frk g\to\R$ be a linear map and set $\frk s:=\{(X,\sigma X):X\in\frk g\}$.
	The vector space $\frk s$ is a Lie subalgebra of $\frk g\rtimes\R$ if and only if $\bigoplus_{j=-2}^{-s}\frk g_j\subset\ker\sigma$.
\end{proposition}
\begin{proof}
	First, we note that $\frk s$ is a Lie algebra if and only if, for all $X,Y\in \frk g$,
	\begin{equation}\label{eq05271343}
	\sigma([X,Y]) + (\sigma X) (\sigma DY) - (\sigma Y) (\sigma DX) = 0 .
	\end{equation}
	
	\ddx
	Suppose $\frk s$ is a Lie algebra, i.e.,~\eqref{eq05271343} holds for all $X,Y\in\frk g$.
	We prove $\bigoplus_{j=-2}^{-s}\frk g_j\subset\ker\sigma$ by induction on $j$.
	If $X,Y\in\frk g_{-1}$, then $DX=X$ and $DY=Y$, thus~\eqref{eq05271343} implies $\sigma([X,Y])=0$.
	Since $\frk g_{-2}=[\frk g_{-1},\frk g_{-1}]$, it follows that $\frk g_{-2}\subset\ker\sigma$.
	Now, suppose that $\frk g_{-k}\subset\ker\sigma$ for $k\ge2$.
	If $X\in\frk g_{-1}$ and $Y\in\frk g_{-k}$, then~\eqref{eq05271343} implies that $\sigma([X,Y])=0$.
	Since $\frk g_{-k-1}=[\frk g_{-1},\frk g_{-k}]$, it follows that $\frk g_{-k-1}\subset\ker\sigma$.
	We conclude that $\bigoplus_{j=-2}^{-s}\frk g_j\subset\ker\sigma$.
	
	\ssx
	Suppose $\bigoplus_{j=-2}^{-s}\frk g_j\subset\ker\sigma$.
	By the bilinearity of the expression, we need to show that~\eqref{eq05271343} holds only when $X\in\frk g_{i}$ and $Y\in\frk g_j$ for some $i$ and $j$.
	Since $\sigma$ is non-zero only on the first layer, the only non-trivial instance of~\eqref{eq05271343} is for $X,Y\in\frk g_{-1}$.
	In this case, $\sigma([X,Y])=0$, and $(\sigma X) (\sigma DY) - (\sigma Y) (\sigma DX) = (\sigma X) (\sigma Y) - (\sigma Y) (\sigma X) =0$.
	Therefore,~\eqref{eq05271343} is satisfied and $\frk s$ is a Lie algebra.
\end{proof}

%%%%%%%%%%%%%%%%%%%%%%%%%%%%%%%%%%%%%%%%%%%%%%%%%%%%%%%%%%%%%%%

\begin{lemma}\label{lem07041801}
	The Lie algebra automorphisms $\phi:\frk p\to\frk p$  such that $\phi(\{0\}\times\R)=\{0\}\times\R$ and $\phi(\frk g_{-1}\times\R) = \frk g_{-1}\times\R$ are exactly those of the form $\phi(X,a)=(\phi_1X,a)$ for some Lie algebra automorphism $\phi_1:\frk g\to\frk g$ that preserves the layers.
\end{lemma}
\begin{proof}
	On the one hand, if $\phi_1:\frk g\to\frk g$ is a Lie algebra automorphism that preserves the layers, then $\phi(X,a)=(\phi_1X,a)$ is clearly a Lie algebra automorphism $\phi:\frk p\to\frk p$ with $\phi(\{0\}\times\R)=\{0\}\times\R$ and $\phi(\frk g_{-1}\times\R) = \frk g_{-1}\times\R$, because $\phi_1D=D\phi_1$.
	
	On the other hand, if $\phi:\frk p\to\frk p$ is a Lie algebra automorphism, then $\phi(\frk g\times\{0\})=\frk g\times\{0\}$ because $\frk g\times\{0\} = [\frk p,\frk p]$.
	Suppose also that $\phi(\{0\}\times\R)=\{0\}\times\R$ and $\phi(\frk g_{-1}\times\R) = \frk g_{-1}\times\R$.
	Then $\phi(X,a)=\phi(X,0)+\phi(0,a) = (\phi_1(X),0) + (0,\phi_2(a))$ and $\phi_1(\frk g_{-1})=\frk g_{-1}$.
	This implies that $\phi_1(\frk g_j)=\frk g_j$ for all $j$, as one can prove by induction on $j$.
	Notice that, for all $X\in\frk g$ and all $a\in\R$,
	\[
	\phi_2(a)D\phi_1X =
	[(0,\phi_2(a)),(\phi_1X,0)] =
	\phi([(0,a),(X,0)]) = \phi(aDX,0) = a\phi_1DX .
	\]
	For every $X\in\frk g_{-1}$, $DX=X$ and $D\phi_1X=\phi_1X$, hence $\phi_2(a) \phi_1X = a\phi_1X$, i.e., $\phi_2(a)=a$.
\end{proof}

\begin{theorem}\label{thm5ea02e33}
	Suppose that $\frk g$ is ultra-rigid, i.e., $\frk p=\frk g\rtimes\R$ is its full Tanaka prolongation.
	The set of all non-isomorphic modifications of $\frk g$ is parametrized by $\frk g_{-1}^*/\R_{>0}$.
	Moreover, all modifications of $\frk g$ in $\frk p$ are solvable and the only nilpotent one is $\frk g$ itself.
\end{theorem}
\begin{proof}
	The set of all modifications of $\frk g$ in $\frk p$ can be identified with $\frk g_{-1}^*$ by Proposition~\ref{prop05271430}, where $\sigma\in\frk g_{-1}^*$ is identified with $\sigma(\sum_jv_j)=\sigma(v_{-1})$ for $\sum_jv_j\in\frk g$ and the modification $\frk s_\sigma:=\{(X,\sigma X):X\in\frk g\}\subset\frk p$.
	 Since $\frk g$ is rigid, by Theorem~\ref{thm05262101} two modifications $\sigma,\tau\in\frk g_{-1}^*$ are isomorphic if and only if there is a Lie algebra automorphism $\phi:\frk p\to\frk p$ with $\phi(\{0\}\times\R)=\{0\}\times\R$ and $\phi(\frk g_{-1}\times\R) = \frk g_{-1}\times\R$ such that $\phi(\frk s_\sigma)=\frk s_\tau$.
	Therefore, by Lemma~\ref{lem07041801}, two modifications $\sigma,\tau\in\frk g_{-1}^*$ are isomorphic if and only if there is a Lie algebra automorphism $\phi_1:\frk g\to\frk g$ such that, for all $X\in\frk g$,
	\[
	(\phi_1X,\sigma X) = 
	(\phi_1X,\tau \phi_1X) ,
	\]
	i.e., $\sigma X=\tau \phi_1X$ for all $X\in\frk g_{-1}$.
	Now, since $\frk g$ is ultrarigid, $\phi_1=\delta_\lambda$ for some $\lambda>0$.
	Therefore, two modifications $\sigma,\tau\in\frk g_{-1}^*$ are isomorphic if and only if there is $\lambda>0$ such that $\sigma=\lambda\tau$.
	
	 Finally, notice that all modifications of $\frk g$ in $\frk p$ are solvable, because $\frk p$ itself is solvable.
        Moreover, the only nilpotent modification is $\frk g$ itself.
        Indeed, if $\frk s\neq\frk g$, then there is $X\in\frk g_{-1}$ with $\sigma X\neq0$, so that,
if $Y\in\frk g_{-s}$ is nonzero,
then $[(X,\sigma X),(Y,0)]= s\sigma X (Y,0)$, where $s$ is the step of $\frk g$.
        Therefore, we obtain that $(Y,0)\in[\frk s,[\dots,[\frk s,\frk s]\dots]]$ for any order of brackets, that is, $\frk s$ is not nilpotent.
\end{proof}

%%%%%%%%%%%%%%%%%%%%%%%%%%%%%%%%%%%%%%%%%%%%%%%%%%%%%%%%%%%%%%%
%%%%%%%%%%%%%%%%%%%%%%%%%%%%%%%%%%%%%%%%%%%%%%%%%%%%%%%%%%%%%%%

\printbibliography

\end{document}